\documentclass[reqno]{amsart}

\usepackage{amsfonts, amssymb,amsthm}
\usepackage{verbatim}       
\usepackage{epsfig}
\usepackage{graphicx}
\usepackage[margin=1.2 in]{geometry}
\usepackage{galois}
\usepackage{dsfont}
\usepackage{amsmath} \allowdisplaybreaks[4]
\usepackage{color}
\usepackage{cite}

\newtheorem{theorem}{Theorem}

\newtheorem{proposition}{Proposition}
\newtheorem{remark}{Remark}
\newtheorem{definition}{Definition}
\newtheorem{corollary}{Corollary}

\newtheorem{example}{Example}
\newtheorem*{acknowledgement}{Acknowledgement}

\def\eps{\varepsilon}

\graphicspath{%
    {converted_graphics/}
    {/}
}
\begin{document}

\title{TOWARD A MATHEMATICAL HOLOGRAPHIC PRINCIPLE}
\thanks{The research of the authors was supported by NSERC grants. The research of Z. Li is also supported
by NNSF of China (No. 11161020 and No. 11361023)}
\subjclass[2000]{37A05, 37H99, 60J05}
\date{\today }
\keywords{multivalued functions; selector; absolutely continuous
invariant measure; mathematical holographic principle}

\author[P. G\'ora]{Pawe\l\ G\'ora }
\address[P. G\'ora]{Department of Mathematics and Statistics, Concordia University,
1455 de Maisonneuve Blvd. West, Montreal, Quebec H3G 1M8, Canada}
\email[P. G\'ora]{pawel.gora@concordia.ca}

\author[Zh. Li]{Zhenyang Li }
\address[Zh. Li]{Department of Mathematics and Statistics, Concordia University,
1455 de Maisonneuve Blvd. West, Montreal, Quebec H3G 1M8, Canada}
\email[Zh. Li]{zhenyangemail@gmail.com}

\author[A. Boyarsky]{Abraham Boyarsky }
\address[A. Boyarsky]{Department of Mathematics and Statistics, Concordia University,
1455 de Maisonneuve Blvd. West, Montreal, Quebec H3G 1M8, Canada}
\email[A. Boyarsky]{abraham.boyarsky@concordia.ca}

\author[H. Proppe]{Harald Proppe}
\address[H. Proppe]{Department of Mathematics and Statistics, Concordia University,
1455 de Maisonneuve Blvd. West, Montreal, Quebec H3G 1M8, Canada}
\email[H. Proppe]{hal.proppe@concordia.ca}

\begin{abstract}
In work started in \cite{GLB13} and
continued in this paper our objective is to study selectors of multivalued functions which have
interesting dynamical properties, such as possessing absolutely continuous
invariant measures. We specify the graph of a multivalued function by means of lower and upper
boundary maps $\tau _{1}$ and $\tau _{2}.$ On these boundary maps we
define a position dependent random map $R_{p}=\{\tau _{1},\tau _{2};p,1-p\},$
which, at each time step, moves the point $x$ to $\tau _{1}(x)$ with
probability $p(x)$ and to $\tau _{2}(x)$ with probability $1-p(x)$. Under
general conditions, for each choice of $p$, $R_{p}$ possesses an
absolutely continuous invariant measure with invariant density $f_{p}.$
Let $\boldsymbol\tau$ be a selector which has invariant
density function $f.$ One of our objectives is to study conditions under
which $p(x)$ exists such that $R_{p}$ has $f$ as its invariant density
function. When this is the case, the long term statistical dynamical
behavior of a selector can be represented by the long
term statistical behavior of a random map on the boundaries of $G.$ We refer
to such a result as a mathematical holographic principle. We present
examples and study the relationship between the invariant densities
attainable by classes of selectors and the random maps based on the
boundaries and show that, under certain conditions, the extreme points of the
invariant densities for selectors are achieved by bang-bang random maps,
that is, random maps for which $p(x)\in \{0,1\}.$
\end{abstract}

\maketitle

\section{Introduction}

A function $\tau :X\rightarrow X$ maps every $x\in X$ to only one point $y=\tau
(x).$ There are applications where this is not the case. For example, in
economics, a consumer's action may not manifest itself in a uniquely
determined process. This is also common in game theory. To study such applications, we need a new
analytical tool, the multivalued function or correspondence as it is called
in the mathematical economics literature \cite{AF90,BL82}. A multivalued function $\Gamma :X\rightrightarrows X$ is a
function from $X$ to the set $2^{X}$ of all subsets of $X$.  The graph of $\Gamma $
is the set: $G=\left\{(x,y)\in X\times X|y\in \Gamma(x)\right\}$. Such maps have
important applications in rigorous numerics \cite{K08}, in economics \cite{A66, H74, D59}, in
dynamical systems \cite{BG97, A80, A00}, chaos synchronization \cite{RALCC01}, and in differential
inclusions \cite{C05}.  Once $\Gamma $ is specified one considers maps $\boldsymbol\tau
:X\rightarrow X$ with $\boldsymbol\tau (x)\in \Gamma (x)$. Such maps are called
 selectors. Establishing the existence of continuous selectors
in topological spaces has been an area of active interest for more than 60
years \cite{M56, S87, AC75, T95, TT96, DP83}. In the setting of chaotic dynamical systems, however,
selectors possessing measure theoretic, rather than topological, properties
are of paramount importance.

In this paper we study selectors of a multivalued function that possess absolutely continuous
invariant measures (acims) and
relate their dynamics to the dynamics of random maps that are based solely
on the boundaries of the graph $G$ defining the multivalued function. We refer
to such a property as holographic. Figure \ref{fig:generic} shows an example of a region $G$ with boundary maps and a selector $\boldsymbol\tau$. Very loosely, the Holographic Principle claims that information in the
interior of a black hole can be described by information on its boundary. In
this note we attempt to establish a basis for an analogous dynamical system result:
let $\boldsymbol\tau$ be a selector with values in $G$, and with probability density function (pdf) $f.$ The main problem we address is: under
what conditions on the boundary maps $\{\tau _{1},\tau _{2}\}$ and on $\boldsymbol\tau$, can we find a probability function $p(x)$ such that the resulting random
map $R_p=\left\{\tau _{1}, \tau _{2}; p, 1-p\right\}$ has $f$ as its invariant pdf? When
this is the case, the long term statistical behavior of the map $\boldsymbol\tau$ inside $G$ is represented by the long term statistical behavior of a
random map defined only on the boundaries of $G$. Such a result qualifies to be referred to as a
mathematical Holographic Principle. Gaining insight into
this dynamical holographic problem is the main objective of this paper.

\begin{figure}[htbp] 
  \centering
  \includegraphics[width=2.55in,height=2.55in,keepaspectratio]{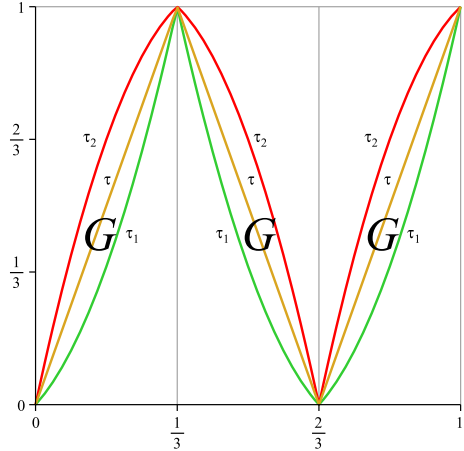}
  \caption{Lower and upper boundary maps $\tau_1$, $\tau_2$ and a selector $\boldsymbol\tau$.}
  \label{fig:generic}
\end{figure}

We shall use random maps \cite{GB03, GB06, GBB03}, where the
probabilistic weights associated with the maps are functions of position as
opposed to being constants in the standard random map framework \cite{RALCC01}. We
define a position dependent random map as follows: let $R_p=\{\tau _{1},\tau
_{2}; p, 1-p\}$, where $\tau _{1}$ and\ $\tau _{2}$ are maps of the unit
interval and $p$ and $1-p$ are position dependent probabilities,
that is, $0\leq p(x)\leq 1$. At each
step, the random map moves the point $x$ to $\tau _{1}(x)$ with
probability $p(x)$ and to $\tau _{2}(x)$ with
probability $1-p(x)$. Of interest to this paper is the following result
\cite{GB03, GB06}: for fixed $\{\tau _{1},\tau _{2}\}$, $R_p$ can have different
invariant pdf's, depending on the choice of the
(weighting) functions $p$. Let $f_{k}$ be an invariant
density of $\tau _{k}$, $k=1,2$. It is shown in \cite{GB03} that for any positive
constants $c_{k}$, $k=1,2$, there exists a system of weighting probability
function $p$ such that the density $f=c_{1}f_{1}+c_{2}f_{2}$ is
invariant under the random map $R_p=\{\tau _{1},\tau _{2}; p, 1-p\},$ where $
p=\frac{c_{1}f_{1}}{c_{1}f_{1}+c_{2}f_{2}}$. (It is assumed that $0/0=0$.)

We now list some of the problems studied in this paper by section:

When can a pdf of a selector be realized as a pdf of a random map based on the boundary maps?
In Section 2 we attempt to gain insight by considering
several examples illustrating that sometimes this is possible but in similar
cases it is not. In Section 3, Theorem \ref{Previous_selector_result} and Corollary \ref{coro_cov} state that any convex combination of pdf's of boundary maps can be realized both as the pdf of a selector and as the pdf of a random map. In Section 4 we use a result from \cite{KCG} to study functional equations whose solutions guarantee that a given selector with pdf $f$ can be achieved by a random map on the boundaries. In Section 5 we present an example where the boundary maps each have a global attracting invariant measure that is singular. Surprisingly, however, we will show that a selector that has Lebesgue invariant measure can be achieved by a position dependent random map on the boundary maps.

In Section 6, we ask: when can a pdf of a random map based on the boundary maps be realized as a pdf of a selector? Theorem \ref{Th_rand_sel} proves that this always holds in the case of piecewise expanding boundary maps
which are of the same monotonicity on partition intervals.

Section 7 establishes a continuity theorem that is used in Section 8, where we give a characterization of the extreme points of the set of pdf's of all random maps based  on the boundary maps. In the last two sections we present a summary of the examples and concluding remarks.

\section{Gaining insight}

In this section we present a number of examples, both positive and negative, as we explore the problem: when a pdf of a selector can be realized as a pdf of a random map based on the boundary maps. Example \ref{quadratic} presents a positive general solution for the case when the selector is the triangle map and the boundary maps have two branches, the first quadratic, the second linear.
\begin{example}\label{quadratic}
In Figure \ref{fig:quadr_tau1_tau2}
a simple region $G$ defined by piecewise monotonic maps consisting of 2
pieces each, where on $[a,1]$ the graphs of both boundary maps are the same
and defined by the slope $\frac{1}{1-a}$ map. The selector we consider
 is the triangle map, whose pdf is $f=1$. We will show that $f$ can be achieved by a position
dependent random map on the boundary maps. In fact, in this example we can
find the exact form of $p(x)$.

\begin{figure}[h]
\begin{minipage}[t]{0.45\linewidth}
\centering
\includegraphics[width=\textwidth]{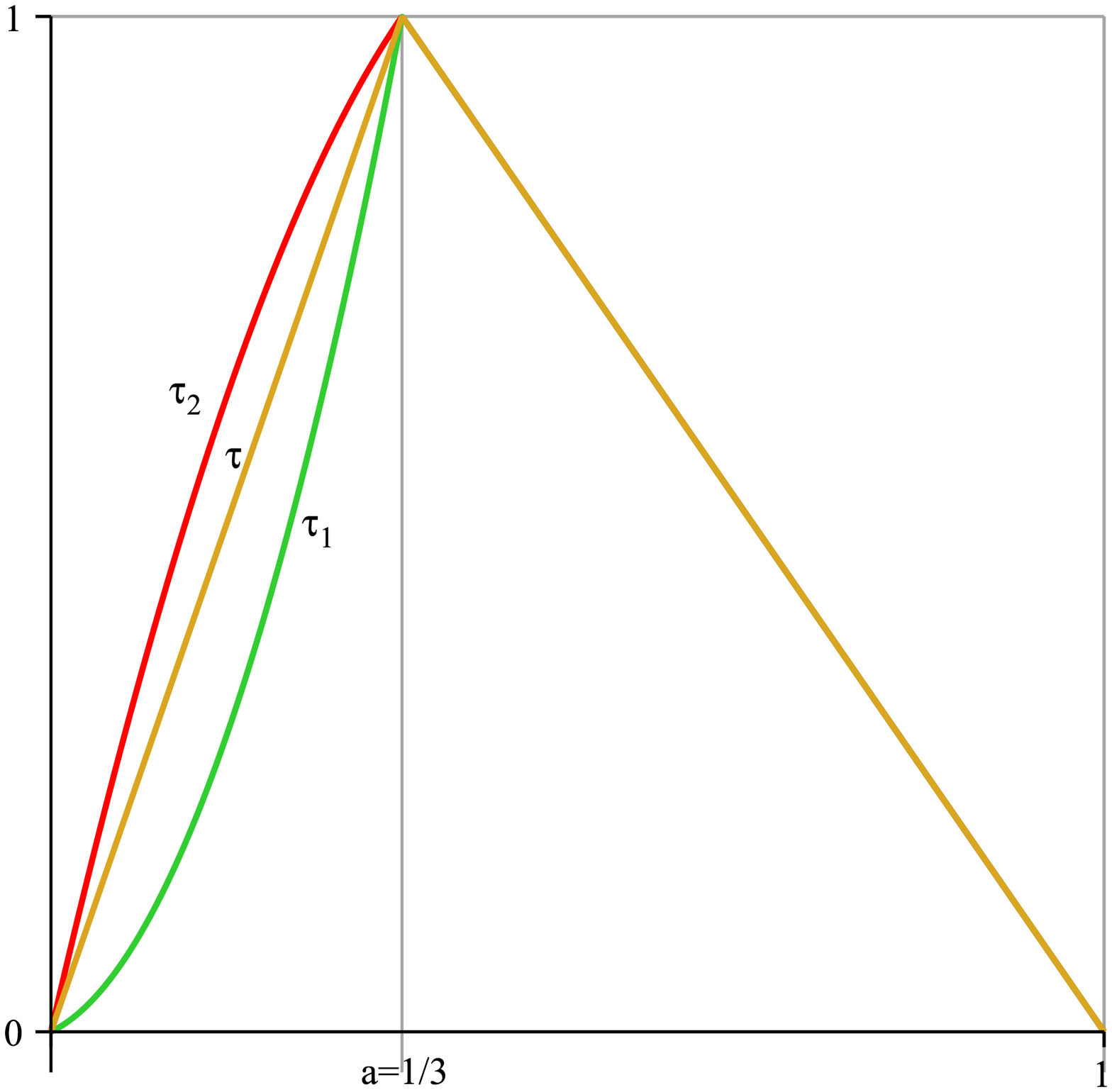}
  \caption{Maps $\tau_1$ and $\tau_2$ from Example  \ref{quadratic}.}
  \label{fig:quadr_tau1_tau2}
\end{minipage}
\hspace{0.5cm}
\begin{minipage}[t]{0.45\textwidth}
\centering
\includegraphics[width=\textwidth]{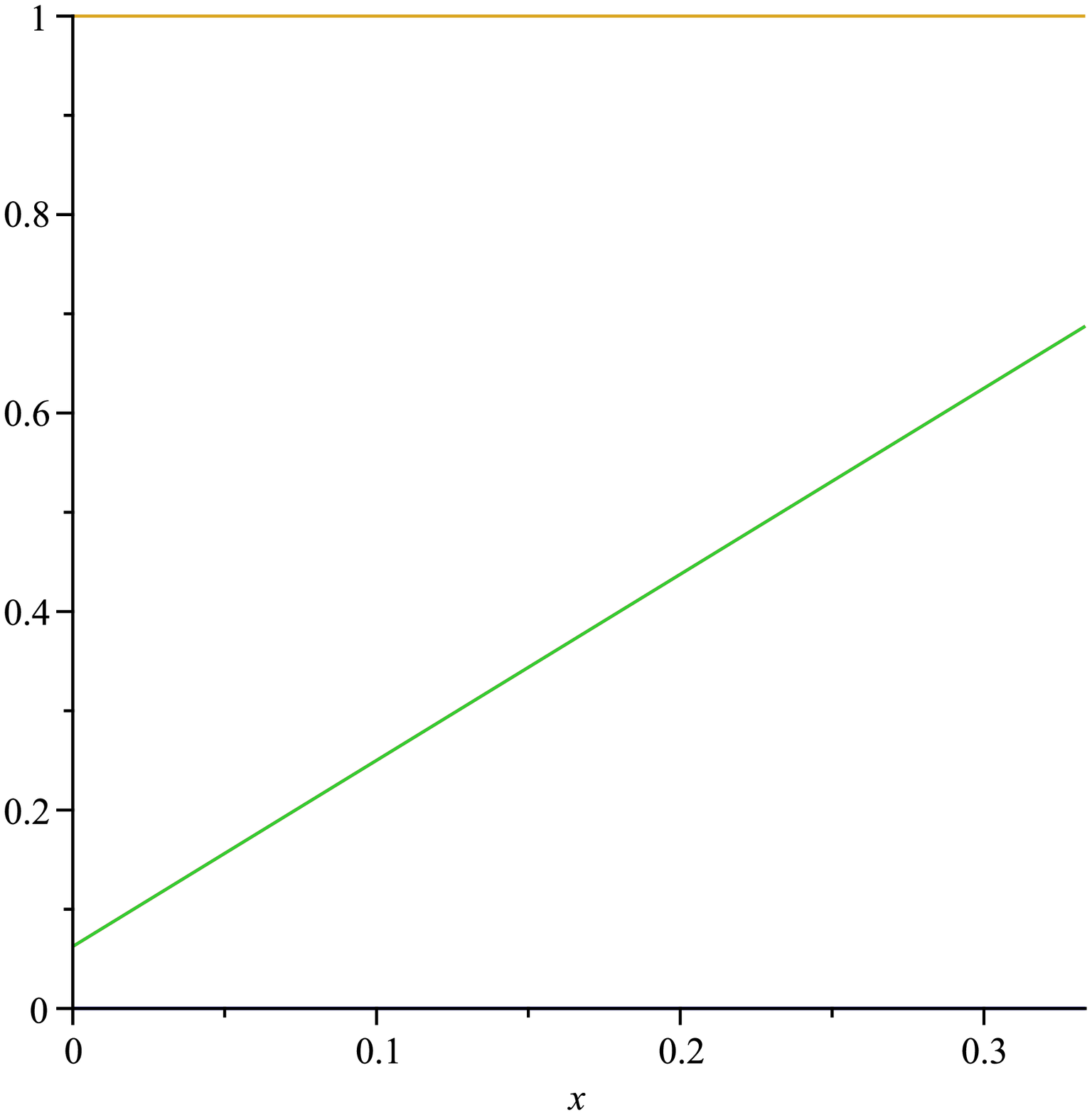}
  \caption{Invariant function $p$ for Example \ref{quadratic}.}
  \label{fig:quadr_p_of_x}
\end{minipage}
\end{figure}


Let us consider $\tau_i:[0,a]\to[0,1]$, $i=1,2$, increasing with $\tau_i(0)=0 $, $\tau_i(a)=1 $ satisfying $\tau_1(x)\le x/a$, $\tau_2(x)\ge x/a$
and define the map $\tau_{21}=\tau_2^{-1}\circ\tau_1:[0,a]\to[0,a]$. $\tau_{21}$ is increasing with $\tau_{21}(0)=0 $, $\tau_{21}(a)=a $, and satisfies $\tau_{21}(x)< x$ on the open interval $(0,a)$.

We consider the special case where $\tau_1$ and $\tau_2$ are quadratic:
$$\tau_1(x)=a_1x^2+b_1x\ \ ,\ \ \tau_2(x)=a_2x^2+b_2x\ ,$$
where $a_i=(1-ab_i)/a^2$, $i=1,2$, $0<b_1<1/a<b_2<2/a$. It follows that
$$\tau_{21}(x)=\frac 1{2(1-ab_2)}\cdot \left(-a^2b_2+\sqrt{a^4b_2^2+4(1-ab_2)[(1-ab_1)x^2+a^2b_1x]}\right)\ .$$

The Frobenius-Perron operator (see \cite{BG97, GB03}) associated with the random map
$R=\{\tau _{1},\tau _{2};p,1-p\}$ is
\begin{equation}\label{quadra_ran}
\left(P_{R}f\right)(x)
=\frac{p(\tau_1^{-1}(x))f(\tau_1^{-1}(x))}{\tau'_1(\tau_1^{-1}(x))}+
\frac{\left(1-p(\tau_2^{-1}(x))\right)f(\tau_2^{-1}(x))}{\tau'_2(\tau_2^{-1}(x))}
+(1-a)f(1-(1-a)x).
\end{equation}
If there exists a probability function $p$ such that $R$ preserves the pdf $f=1$, i.e. $\left(P_{R}1\right)(x)=1\ a.e.$, then
it follows from Equation (\ref{quadra_ran}) that
\begin{equation}\label{main_eq_qua_V1}
p(\tau_1^{-1}(x))=p(\tau_2^{-1}(x))\frac{\tau'_1(\tau_1^{-1}(x))}{\tau'_2(\tau_2^{-1}(x))}-
\frac{\tau'_1(\tau_1^{-1}(x))}{\tau'_2(\tau_2^{-1}(x))}+a\tau'_1(\tau_1^{-1}(x)).
\end{equation}
After introducing the new variable $y=\tau_1^{-1}(x)$ and replacing $y$ with $x$, equation (\ref{main_eq_qua_V1}) becomes the functional equation
\begin{equation}\label{main_eq_qua}
p(x)=p(\tau_{21}(x))\cdot\tau'_{21}(x)-\left[\tau_{21}'(x)-a\cdot\tau_1'(x)\right]\ .
\end{equation}
Now we need to verify that $p$ satisfies $$0\leq p\leq 1\ .$$

The linear function $p(x)=A+Bx$ satisfies (\ref{main_eq_qua}) if we let
$$ A=\frac{b_1(1-ab_2)}{b_1-b_2}\ \  ,\ \ B=\frac{2(1-ab_1)(1-ab_2)}{(b_1-b_2)a^2}\ .$$

 Let $a=1/3$, $b_1=0.5$ and $b_2=4.5$. Then $A=0.0625$ and $B=1.875$. The
maps $\tau_1$, $\tau_2$ and the function $p$ are shown in Figures \ref{fig:quadr_tau1_tau2}
 and \ref{fig:quadr_p_of_x}, respectively. Note that $p(x)$ is defined on $[0,a]$, and is arbitrary on $(a,1]$
  since both maps are the same on $(a,1]$.

 \end{example}

Example \ref{fivexmod1} presents a positive solution for the specific case of a selector and semi-Markov boundary maps with five branches.
\begin{example}\label{fivexmod1}
Consider boundary maps which are piecewise linear
semi-Markov maps as in \cite{GB06}. Consider $\tau_1$ and $\tau_2$ as defined below and shown in Fig.\ref{fivex_mod1}. We can show that the selector defined by $\boldsymbol\tau (x)=5x\ \rm{mod\,}1$, which has
pdf $f=1,$ can be attained by a position dependent random map.
\begin{equation*}
\tau_1(x)=\begin{cases}2x \ ,\ &\text{\ for\ } \ 0\le x\le 0.1\ ;\\
                   8x-0.6 \ ,\ &\text{\ for\ } \ 0.1\le x< 0.2\ ;\\
                   5x\ \rm{mod\,}1 \ ,\ &\text{\ for\ } \ 0.2\le x\leq 1\ .
\end{cases}\ \ \ \ \
\tau_2(x)=\begin{cases} 8x \ ,\ &\text{\ for\ } \ 0\le x\le 0.1\ ;\\
                   2x+0.6 \ ,\ &\text{\ for\ } \ 0.1\le x< 0.2\ ;\\
                   5x\ \rm{mod\,}1 \ ,\ &\text{\ for\ } \ 0.2\le x\leq 1\ .
\end{cases}
\end{equation*}
\begin{figure}[htp]
\centering
\includegraphics[width=2.55in,height=2.55in,keepaspectratio]{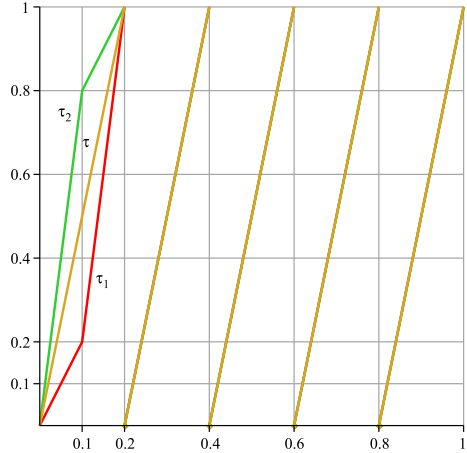}
\caption{A selector whose pdf can be achieved by a random map defined by the boundary maps.}
  \label{fivex_mod1}
\end{figure}

Let us consider the random map $R=\{\tau _{1},\tau _{2};p,1-p\}$, where
\begin{equation*}
p(x)=\begin{cases}0.2 \ ,\ &\text{\ for\ } \ 0\le x< 0.1\ ;\\
                   0.8 \ ,\ &\text{\ for\ } \ 0.1\le x\le 0.2\ ,
\end{cases}
\end{equation*}
and arbitrary on the interval $(0.2,1]$. Let us use the matrix operator \cite{BG97} which corresponds to the Frobenius-Perron operator associated with the random map $R$. Denote the induced matrices of $\tau _{1}$ and $\tau _{2}$
by $M_1$ and $M_2$, respectively. We denote the probability function $p$ by the vector $[p_1, p_2, \ldots, p_6]$, where $p_1=0.2$, $p_2=0.8$, and $p_k$, $k=3, \ldots, 6$, are arbitrary. Let $q$ equal $[1-p_1, 1-p_2, \ldots, 1-p_6]$.
 Now the induced matrix of the map $R$ \cite{GB06} is \[M=\text{diag}(p)M_1+\text{diag}(q)M_2\ .\]
 Thus,
\begin{equation*}
M=\left[ \begin {array}{cccccc}  1/5& 1/5&
 1/10& 1/10& 1/10&0\\ \noalign{\medskip}0&0&
 1/10& 1/10& 1/10& 1/5
\\ \noalign{\medskip}1/5&1/5&1/5&1/5&1/5&1/5\\ \noalign{\medskip}1/5&1
/5&1/5&1/5&1/5&1/5\\ \noalign{\medskip}1/5&1/5&1/5&1/5&1/5&1/5
\\ \noalign{\medskip}1/5&1/5&1/5&1/5&1/5&1/5\end {array} \right].
\end{equation*}
It can be checked that the left invariant eigenvector is $[1,1,1,1,1,1]$, which implies the random map $R$
preserves Lebesgue measure, the acim of the selector $\boldsymbol\tau(x)=5x\ \rm{mod\,}1$.
\end{example}

Example \ref{bendedfivexmod1} shows that with the same boundary maps as in Example \ref{fivexmod1} and another selector (which has an invariant pdf), it is impossible.
\begin{example}\label{bendedfivexmod1}
In this example, $\tau_1$ and $\tau_2$ are the same as in Example \ref{fivexmod1}, but the selector $\boldsymbol\tau$ is defined by:
\begin{equation*}
\boldsymbol\tau(x)=\begin{cases} 6x \ ,\ &\text{\ for\ } \ 0\le x\le 0.1\ ;\\
                   4x+0.2 \ ,\ &\text{\ for\ } \ 0.1\le x< 0.2\ ;\\
                   5x\ \rm{mod\,}1 \ ,\ &\text{\ for\ } \ 0.2\le x\leq 1\ .
\end{cases}
\end{equation*}

\begin{figure}[htp]
\centering
\includegraphics[width=2.55in,height=2.55in,keepaspectratio]{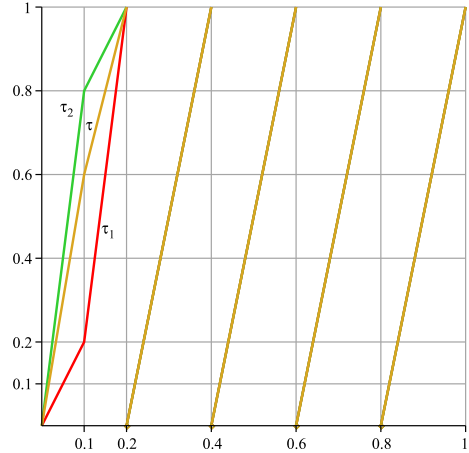}
\caption{A selector whose pdf cannot be achieved by a random map defined by the boundary maps.}
\label{bended_5xmod1}
\end{figure}

The graph is shown in Fig.\ref{bended_5xmod1}.
The invariant density of $\boldsymbol\tau$ is
\begin{equation}\label{den_desired}
f(x)=\frac{30}{31}\chi_{_{[0,0.6]}}(x)+\frac{65}{62}\chi_{_{[0.6,1]}}(x).\end{equation}
Let
\begin{equation*}
p(x)=\begin{cases} p_1(x) \ ,\ &\text{\ for\ } \ 0\le x\le 0.2\ ;\\
                   1 \ ,\ &\text{\ for\ } \ 0.2< x\le 1\ .
\end{cases}
\end{equation*}
Let $q(x)=1-p(x)$, for $x\in[0,1]$.
We define a random map $R=\{\tau_1,\tau_2;\ p(x),q(x)\}$ whose Frobenius-Perron operator \cite{GB03} is given by
\begin{eqnarray*}
\left(P_{R}f\right)(x)
&=&\frac{p_1(\frac x2)f(\frac x2)}{2}\chi_{_{[0,0.2]}}(x)+\frac{p_1(\frac{x}{8}+\frac{3}{40})f(\frac{x}{8}+\frac{3}{40})}{8}\chi_{_{[0.2,1]}}(x)\\
& & +\frac{\left(1-p_1(\frac x8)\right)f(\frac x8)}{8}\chi_{_{[0,0.8]}}(x)+\frac{\left(1-p_1(\frac{x}{2}-\frac{3}{10})\right)f(\frac{x}{2}-\frac{3}{10})}{2}\chi_{_{[0.8,1]}}(x)
\\
& &
+\frac 15\sum_{i=3}^6f\left(\tau_{1,i}^{-1}(x)\right).
\end{eqnarray*}
If $f(x)$ defined in (\ref{den_desired}) is an invariant density of $R$, then we must have
\[\left(P_{R}f\right)(x)=f(x),\ a.e.\]
Thus, we obtain
\begin{eqnarray*}
&&\frac{15}{31}p_1(\frac x2)\chi_{_{[0,0.2]}}(x)+\frac{15}{124}p_1(\frac{x}{8}+\frac{3}{40})\chi_{_{[0.2,1]}}(x)\\
& &-\frac{15}{124}p_1(\frac x8)\chi_{_{[0,0.8]}}(x)-\frac{15}{31}p_1(\frac{x}{2}-\frac{3}{10})\chi_{_{[0.8,1]}}(x)\\
&=&\frac{5}{124}\chi_{_{[0,0.6]}}(x)+\frac{15}{124}\chi_{_{[0.6,0.8]}}(x)-\frac{15}{62}\chi_{_{[0.8,1]}}(x).
\end{eqnarray*}
Now, if $x\in[0.6,0.8]$, we have
\begin{equation}\label{prob1}p_1(\frac x8 +\frac{3}{40})-p_1(\frac x8)=1.\end{equation}
Thus, $p_1(z)=0$ for $z\in [0.075, 0.1]$ while $p_1(z)=1$ for $z\in [0.15, 0.175]$. Now consider $x\in [0.9, 0.95]$, we have
\begin{equation}\label{prob2}p_1(\frac{x}{8}+\frac{3}{40})=4p_1(\frac x2 -\frac{3}{10})-2.\end{equation}
Since $\frac x2 -\frac{3}{10}\in [0.15, 0.175]$, it follows from Equation (\ref{prob1}) that
\[p_1(\frac{x}{8}+\frac{3}{40})=2,\ \ x\in[0.9, 0.95],\]
which is a contradiction to $p_1$ being a probability function.
\end{example}

In summary, the ability to represent the long term statistical dynamics of
an arbitrary selector by the long term statistical dynamics of a position
dependent random map based on the boundaries is a very sensitive matter.

\section{Classes of selectors which can be ``represented" by random maps
on the boundaries}\label{pre_sel}

Without loss of generality, we assume that $\tau_1$ and $%
\tau_2$ have a common partition: $a_0=0<a_1<a_2<\cdots<a_{m}=1$. Let $I_j=[a_{j-1},a_j]$, $%
j=1,2,3,\ldots,m$. On each interval $I_j$, $\tau_{1,j}:=\tau_1|_{_{I_j}}$
and $\tau_{2,j}:=\tau_2|_{_{I_j}}$ share the same monotonicity, where we
understand $\tau_{1,j}$ and $\tau_{2,j}$ as the natural extensions of pieces
of $\tau_1$ and $\tau_2$, respectively.

Let us assume that $\tau_1$ and $\tau_2$ preserve measures $\mu_1$ and $\mu_2$, respectively.
Let $F^{(1)},F^{(2)}$ be the distribution functions of measures $\mu_1$, $\mu_2$,
respectively, i.e. $F^{(i)}(x)=\mu_i([0,x])$, $i=1,2$. Let $\mu=\lambda\mu_1+(1-%
\lambda)\mu_2$, $0<\lambda<1$, and let $F$ be the distribution function of $%
\mu$, i.e. $F(x)=\mu([0,x])$. Let $\mathcal{T}$ be the class of piecewise $C^{1}$ maps from $I$ into $I$ whose graphs belong to $G$. In \cite{GLB13}, the following theorem is proved.
\begin{theorem}\label{Previous_selector_result}
Let $\tau_1, \tau_2 \in \mathcal{T}$ and assume they possess continuous invariant distribution functions $F^{(1)}$ and $F^{(2)}$. Then,
for any convex combination $F=\lambda F^{(1)}+(1-\lambda) F^{(2)}$, $0<\lambda<1$,
there exists a piecewise monotonic selector $\boldsymbol\tau$, $\tau_1\le\boldsymbol\tau\le\tau_2$, preserving the distribution function $F$.
\end{theorem}

Moreover, the formula of each piece of $\boldsymbol\tau$ is given.
For any interval $[a,b]\subseteq [0,1]$, given a monotone continuous
function $h:[a,b]\rightarrow [0,1]$, we define its extended inverse as
follows. Let
\begin{equation*}
h^{\max}=\max\left\{h(x)|x\in[a,b]\right\},\text{ and }
h^{\min}=\min\left\{h(x)|x\in[a,b]\right\}.
\end{equation*}
Depending on $h$ is increasing or decreasing, its extended inverse is defined as
\begin{equation*}
\overline{h^{-1}}(x)=%
\begin{cases}
a \ , \text{\ for\ } \ x\in[0,h^{\min}]\ ; \\
h^{-1}(x) \ , \text{\ for\ } \ x\in[h^{\min},h^{\max}]\ ; \\
b \ , \text{\ for\ } \ x\in[h^{\max},1]\ ;
\end{cases}\ \ \ \
\overline{h^{-1}}(x)=
\begin{cases}
b \ , \text{\ for\ } \ x\in[0,h^{\min}]\ ; \\
h^{-1}(x) \ , \text{\ for\ } \ x\in[h^{\min},h^{\max}]\ ; \\
a \ , \text{\ for\ } \ x\in[h^{\max},1]\ , \\
\end{cases}
\end{equation*}
respectively. We define the extended inverse of each branch of $\boldsymbol\tau$ by
\begin{equation}  \label{formula_eta}
\overline{\boldsymbol\tau_{j}^{-1}}(x)=F^{-1}\left(\lambda F^{(1)}(\overline{%
\tau_{1,j}^{-1}}(x))+(1-\lambda) F^{(2)}(\overline{\tau_{2,j}^{-1}}%
(x))\right)\, ,
\end{equation}
where $j=1,2,3,\ldots,m$. $\boldsymbol\tau$ defined in this way, after the vertical
segments are removed, has the same number of branches as $\tau_1$ and $\tau_2$.

In equation (\ref{formula_eta}), if $\lambda=0$, then our selector $\boldsymbol\tau$
is $\tau_2$; if $\lambda=1$, then the selector $\boldsymbol\tau$
is $\tau_1$. Therefore, as $\lambda$ varies from 0 to 1, we have a
decomposition of $G$ into a pairwise disjoint union of curves, which can be
though of as a kind of ``foliation" of $G.$

The above Theorem \ref{Previous_selector_result} and the Theorem 4 in \cite{GB03} together
establish the following corollary.

\begin{corollary}\label{coro_cov}
Let $\boldsymbol\tau$ be the selector that preserves the density $f=\lambda f_1+(1-\lambda)f_2$. Then there exists a probability function $p$ such that random map $R=\left\{\tau _{1},\tau _{2};p,1-p\right\}$ also preserves $f$.
\end{corollary}

From Theorem 6 in \cite{GLB13} and Corollary \ref{coro_cov}, it follows that we can consider a more general case:
\begin{theorem}
Let $\tau_i=h_i^{-1}\circ\Lambda\circ h_i$, $i=1,2$, where $\Lambda$ is a triangle map, $h_1$ and $h_2$ are two diffeomorphisms. Define a diffeomorphism $h=\lambda h_1+
(1-\lambda) h_2$. Then the map $\boldsymbol\tau=h^{-1}\circ\Lambda\circ h$ is a selector between $\tau_1$ and $\tau_2$,
and thus its pdf can be preserved by a random map defined by
$\tau_1$ and $\tau_2$.
\end{theorem}

\section{Functional equations for probability functions}

\subsection{For Lebesgue measure}

Let us recall the setting for the boundary maps as in Example \ref{quadratic} and the functional equation (\ref{main_eq_qua}), the second branch of the boundary maps is linear, but now we do not require the first branches of our
maps to be quadratic.

We choose $0<\widehat{a}$ such that $$\tau_{21}'(x)\le \sigma<1\ \  \text{on}\ \ [0,\widehat{a}]\ ,\ \widehat{a}<a\ .$$
Note that $\sigma$ depends on the choice of $\widehat{a}$. Let us define an affine operator
$$\left(\mathcal{P} p\right)(x)=p(\tau_{21}(x))\cdot\tau'_{21}(x)-\left[\tau_{21}'(x)-a\cdot\tau_1'(x)\right]\ .$$
The most natural space on which to consider the operator is $L^\infty[0,a]$, the space of bounded functions on the interval $[0,a]$.
Note that $\mathcal{P}$ preserves the integral on $[0,1]$ with respect to Lebesgue measure.

First, we will consider $\mathcal{P}$ on the smaller space $L^\infty[0,\widehat{a}] $.

\begin{proposition}\label{Prop_p1}
$\mathcal{P}$ is a contraction on $L^\infty[0,\widehat{a}]$.
\end{proposition}
\begin{proof}  We have
\begin{eqnarray*}\|\mathcal{P} p-\mathcal{P} q\|&=&\sup_{x\in [0,\widehat{a}]}|p(\tau_{21}(x))\cdot\tau'_{21}(x)-q(\tau_{21}(x))\cdot\tau'_{21}(x)|\\
&\le&
\sup_{x\in [0,\widehat{a}]}|p(\tau_{21}(x))-q(\tau_{21}(x))|\cdot\sup_{x\in [0,\widehat{a}]}|\tau'_{21}(x)|\le\lambda \| p- q\|\ .
\end{eqnarray*}
\end{proof}

\begin{proposition} Equation (\ref{main_eq_qua}), which applies to selectors with pdf$\equiv 1$, considered on $[0,\widehat{a}]$ has exactly one solution.
\end{proposition}
\begin{proof}This follows by Banach's contraction principle.
\end{proof}

\begin{proposition} Using Theorem 2.2.1 in \cite{KCG}, the solution on $[0,\widehat{a}]$ is given by
\begin{equation}\label{sol}
p(x)=\sum_{n=0}^\infty B(\tau_{21}^n(x))\cdot (\tau_{21}^n)'(x) \ ,
\end{equation}
where $B(x)=-[\tau_{21}'(x)-a\tau_1'(x)]=a\tau_1'(x)-\tau_{21}'(x) $.
\end{proposition}
\begin{proof} The solution $p$ can be obtained as a limit
of the functions $\mathcal{P}^n(1)$ \cite{BG97}. The  series (\ref{sol}) converges since it is dominated by a geometric series in $\sigma$.

On the other hand, we can just substitute the series (\ref{sol}) into (\ref{main_eq_qua}) to get:
$$ \sum_{n=0}^\infty B(\tau_{21}^n(x))\cdot (\tau_{21}^n)'(x) =
\sum_{n=0}^\infty B(\tau_{21}^n(\tau_{21}(x)))\cdot (\tau_{21}^n)'(\tau_{21}(x))\cdot\tau'_{21}(x)+B(x)\ ,$$
obtaining the desired equality.
\end{proof}

\textbf{Remarks on the series (\ref{sol}):} On every subinterval $[0,s]\subset [0,a)$ the convergence is uniform (for $\widehat{a}<s$ as well).
The series is divergent at $a$ (unless $B(a)=0$) so the value of $p(a)$ should be obtained using $p(\tau_{21}(a))$ and equation (\ref{main_eq_qua}).
Let us define
$$f_k(x)=\sum_{n=0}^k B(\tau_{21}^n(x))\cdot (\tau_{21}^n)'(x) \ \ ,\ \ k=0,1,2,\dots$$
Note that $\int_0^a f_k(x)dx =0$ for all $k\ge 0$, while for the limit $p$ we usually have $\int_0^a p(x)dx >0$. This,
 in general, means that the functions $f_k$ are not uniformly integrable (i.e., for arbitrary constant $M> 0$
and arbitrary $\eps>0$ we can find $k\ge 0$ such that $\int_{[a-\eps,a]} |f_k(x)| dx >M$) and the series in (\ref{sol}) does not converge in $L^1[0,a]$.

\begin{proposition}\label{Prop_p4}
Once the solution $p$ is known on $[0,\widehat{a}]$ it is uniquely extended to $[0,a)$.
The solution is still described by the formula (\ref{sol}). The value of $p(a)$ is obtained using $p(\tau_{21}(a))$ and equation (\ref{main_eq_qua}).
\end{proposition}
\begin{proof} First note that the function $\tau_{21}^{-1}$ has two fixed points, $0$ and $a$. Since $\tau_{21}(x)<x$ on $(0,a)$, the sequence $\{\tau_{21}^{-n}(\widehat{a})\}$ is strictly increasing and converges to the fixed point $a$ as $n\to\infty$. We extend $p$ using equation (\ref{main_eq_qua}): from $[\tau_{21}(\widehat{a}),\widehat{a})$
to $[\widehat{a},\tau_{21}^{-1}(\widehat{a}))$, then to $[\tau_{21}^{-1}(\widehat{a}),\tau_{21}^{-2}(\widehat{a}))$, then to
$[\tau_{21}^{-2}(\widehat{a}),\tau_{21}^{-3}(\widehat{a}))$, etc.

The solution is uniquely described by the formula (\ref{sol}).
\end{proof}
\begin{remark}
Note that $p$ is uniquely determined and independent of the choice of $\widehat{a}$.
\end{remark}

Example \ref{posit_ex} shows a successful application of Propositions \ref{Prop_p1}-\ref{Prop_p4} to find the probability function defining a random map with a required pdf. Example \ref{neg_ex} shows that the method of Propositions \ref{Prop_p1}-\ref{Prop_p4} sometimes fails. The function produced is not a probability function.
\begin{example}\label{posit_ex}
We set $a=1/2$. We consider $\tau_1(x)=2x^2+x$, $\tau_2(x)=-2x^2+3x$,
$\tau_{21}(x)=\tau_2^{-1}(\tau_1(x))=3/4-(1/4)\sqrt{9-16x^2-8x}$. Using Maple 13 we were able to guess that the solution is $p(x)=x+1/4$.
We illustrate the statements with a number of pictures depicted in Figs.\ref{fig:tau1_tau2}-\ref{fig:approximations50}.

\begin{figure}[h]
\begin{minipage}[t]{0.45\linewidth}
\centering
\includegraphics[width=\textwidth]{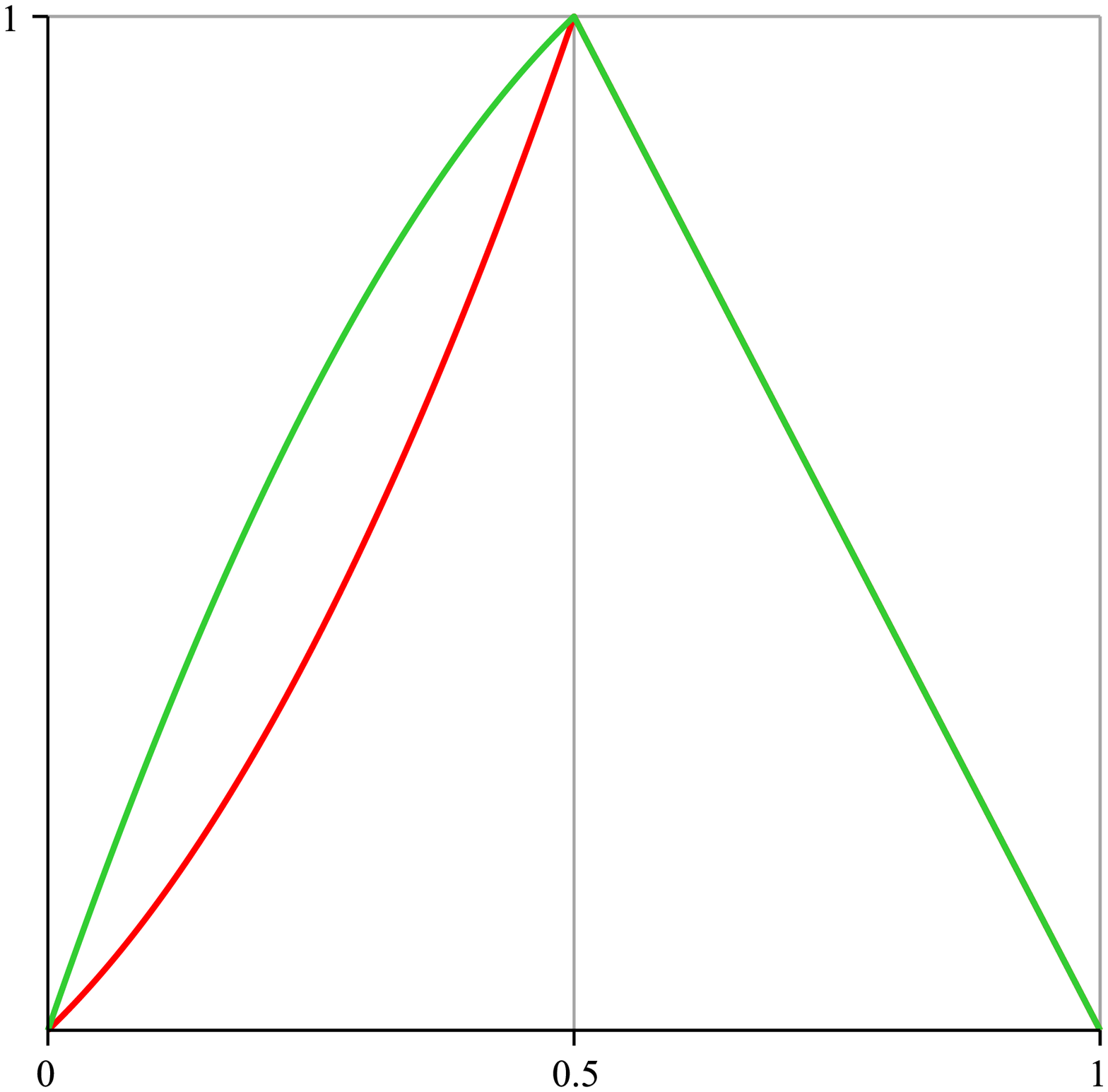}
\caption{Maps $\tau_1$ and $\tau_2$ for Example \ref{posit_ex}.}
\label{fig:tau1_tau2}
\end{minipage}
\hspace{0.5cm}
\begin{minipage}[t]{0.45\textwidth}
\centering
\includegraphics[width=\textwidth]{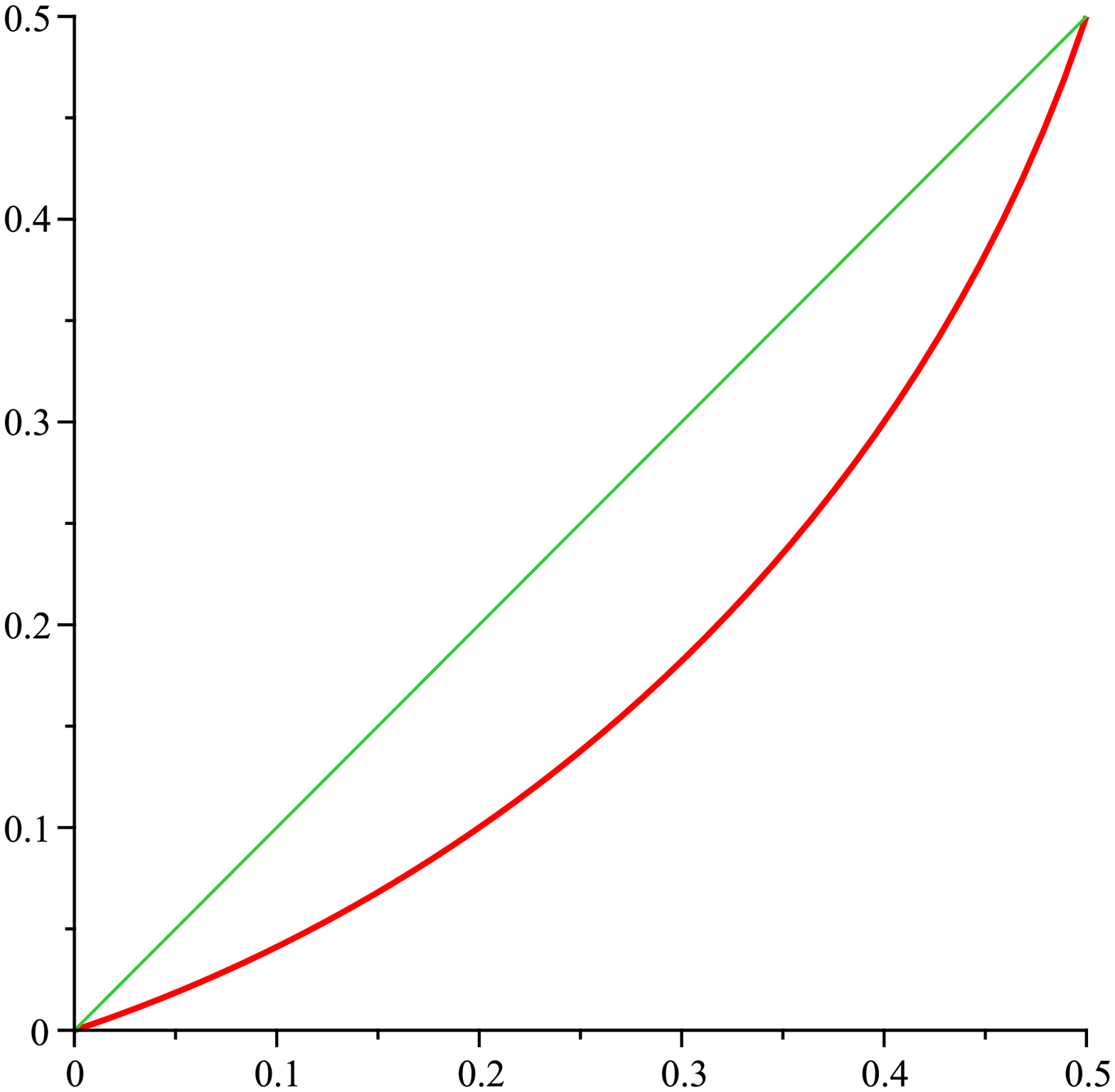}
\caption{Map $\tau_{21}$ for Example \ref{posit_ex}.}
\label{fig:tau21}
\end{minipage}
\end{figure}

\begin{figure}[h]
\begin{minipage}[t]{0.45\linewidth}
\centering
\includegraphics[width=\textwidth]{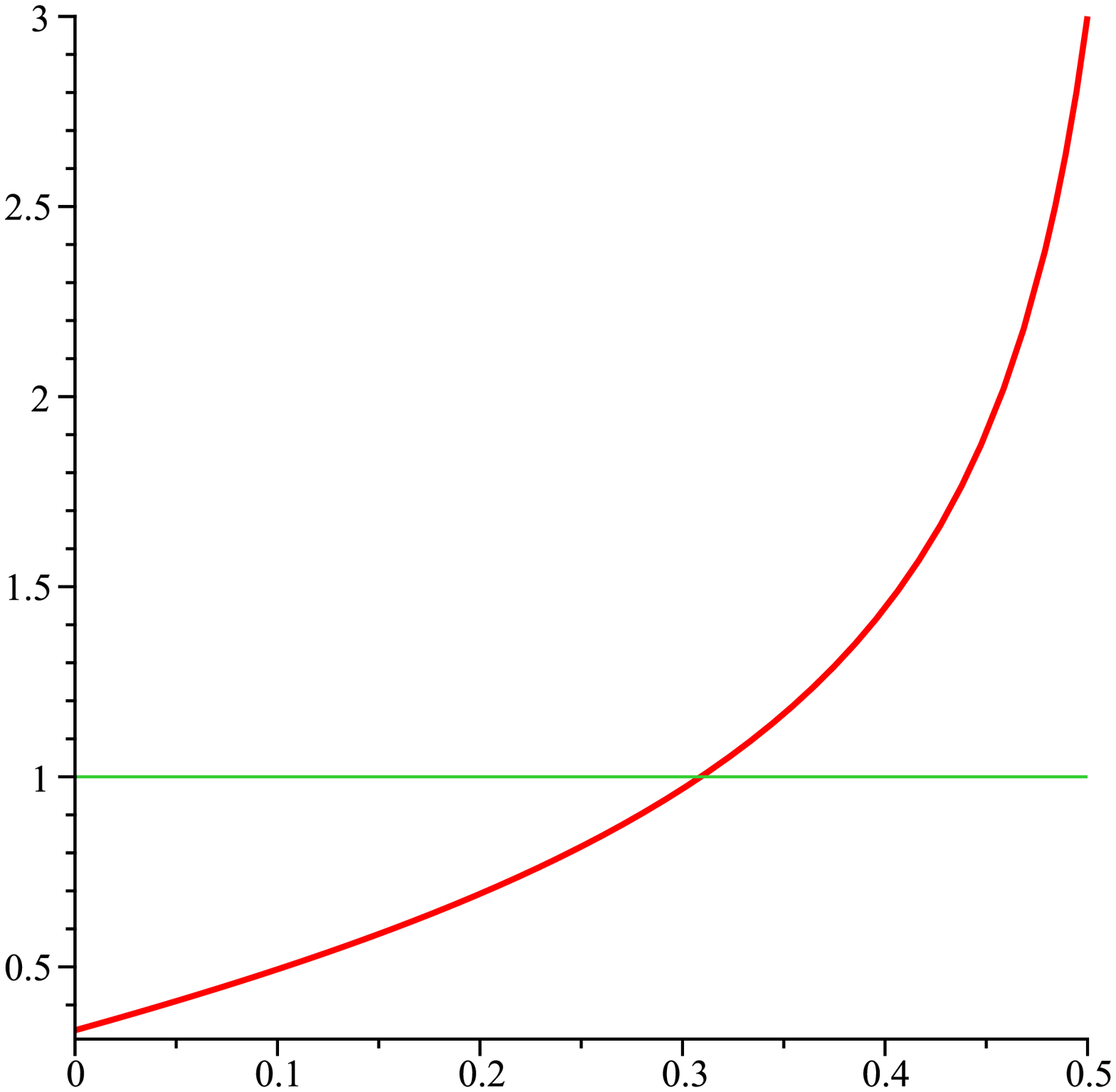}
\caption{Derivative $\tau_{21}'$ for Example \ref{posit_ex}.}
\label{fig:tau21prime}
\end{minipage}
\hspace{0.5cm}
\begin{minipage}[t]{0.45\linewidth}
\centering
\includegraphics[width=\textwidth]{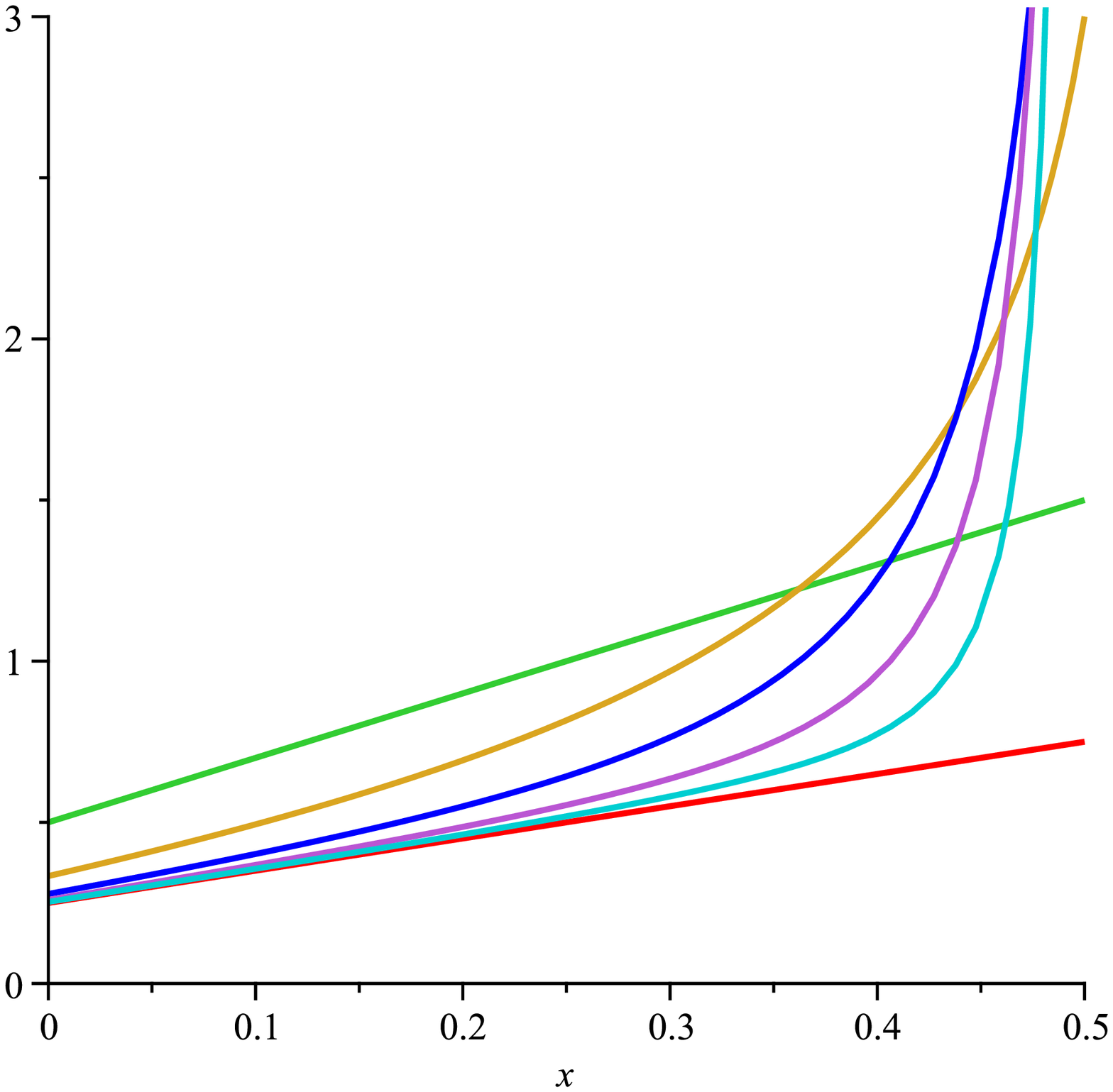}
\caption{First five iterations of approximating the invariant solution $p$ on $[0,0.5]$. Cut off at $3$ for Example \ref{posit_ex}.}
\label{fig:approximations50}
\end{minipage}
\end{figure}
\end{example}

\begin{example}\label{neg_ex}
In this example the solution $p$ is not between $0$ and $1$.
We set $a=1/5$ and consider
$$\tau_1(x)=\begin{cases} (4/3)x^2+1.365128205 x&\ ,\ \text{for}\ 0\le x< 0.13\ ;\\
                    138.8110936 x^2  -34.37908948 x+   2.323374150 &\ ,\ \text{for}\ 0.13\le x\le a\ ,
\end{cases}
$$
 and $\tau_2(x)=1-\tau_1(1/5-x)$,
$\tau_{21}(x)=\tau_2^{-1}(\tau_1(x))$.
We illustrate the statements with a number of pictures depicted in Figs.\ref{fig:neg_tau1_tau2}-\ref{fig:neg_approximations20}.

\begin{figure}[h]
\begin{minipage}[t]{0.45\linewidth}
\centering
\includegraphics[width=\textwidth]{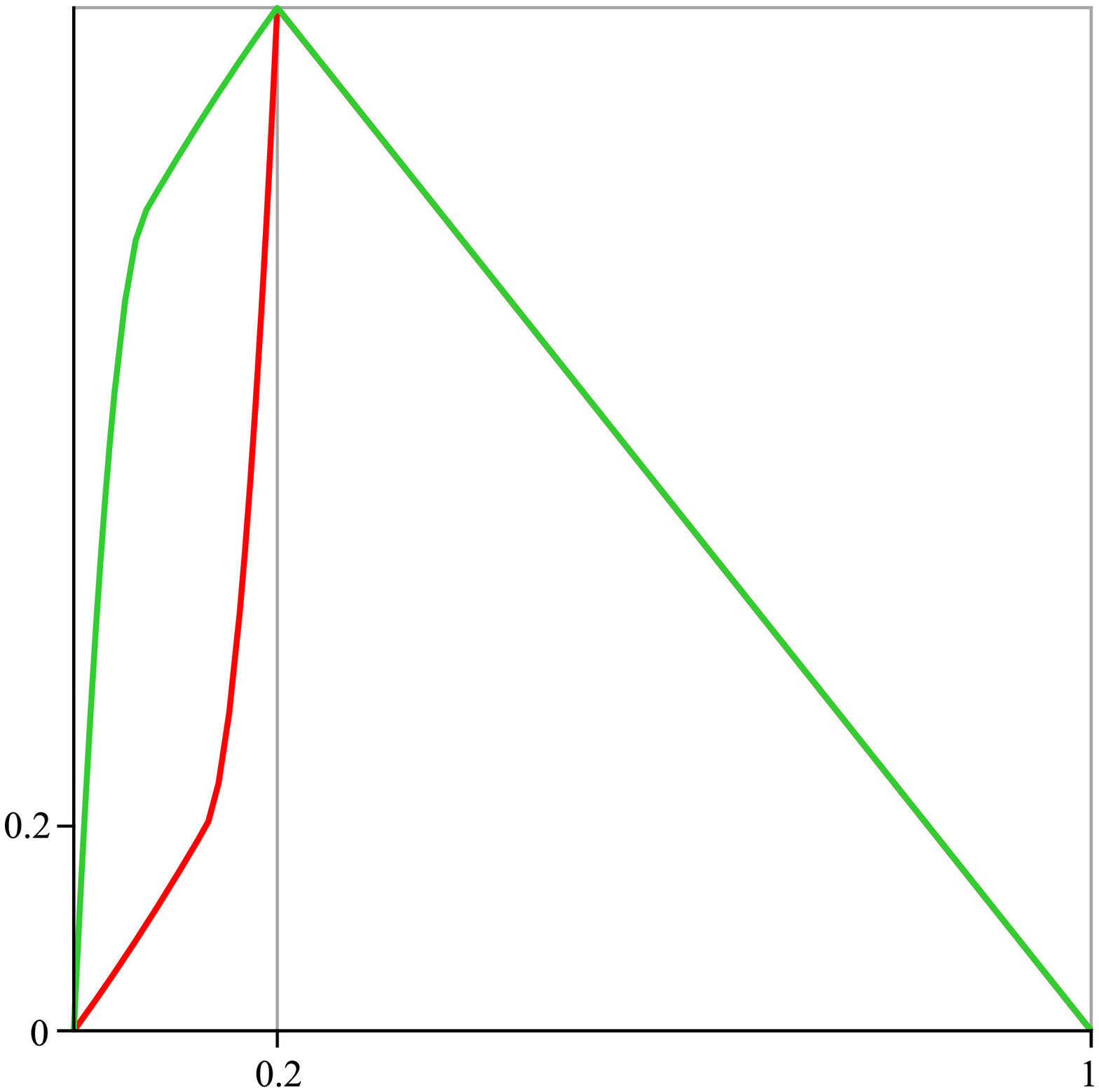}
\caption{Maps $\tau_1$ and $\tau_2$ for Example \ref{neg_ex}.}
\label{fig:neg_tau1_tau2}
\end{minipage}
\hspace{0.5cm}
\begin{minipage}[t]{0.45\linewidth}
\centering
\includegraphics[width=\textwidth]{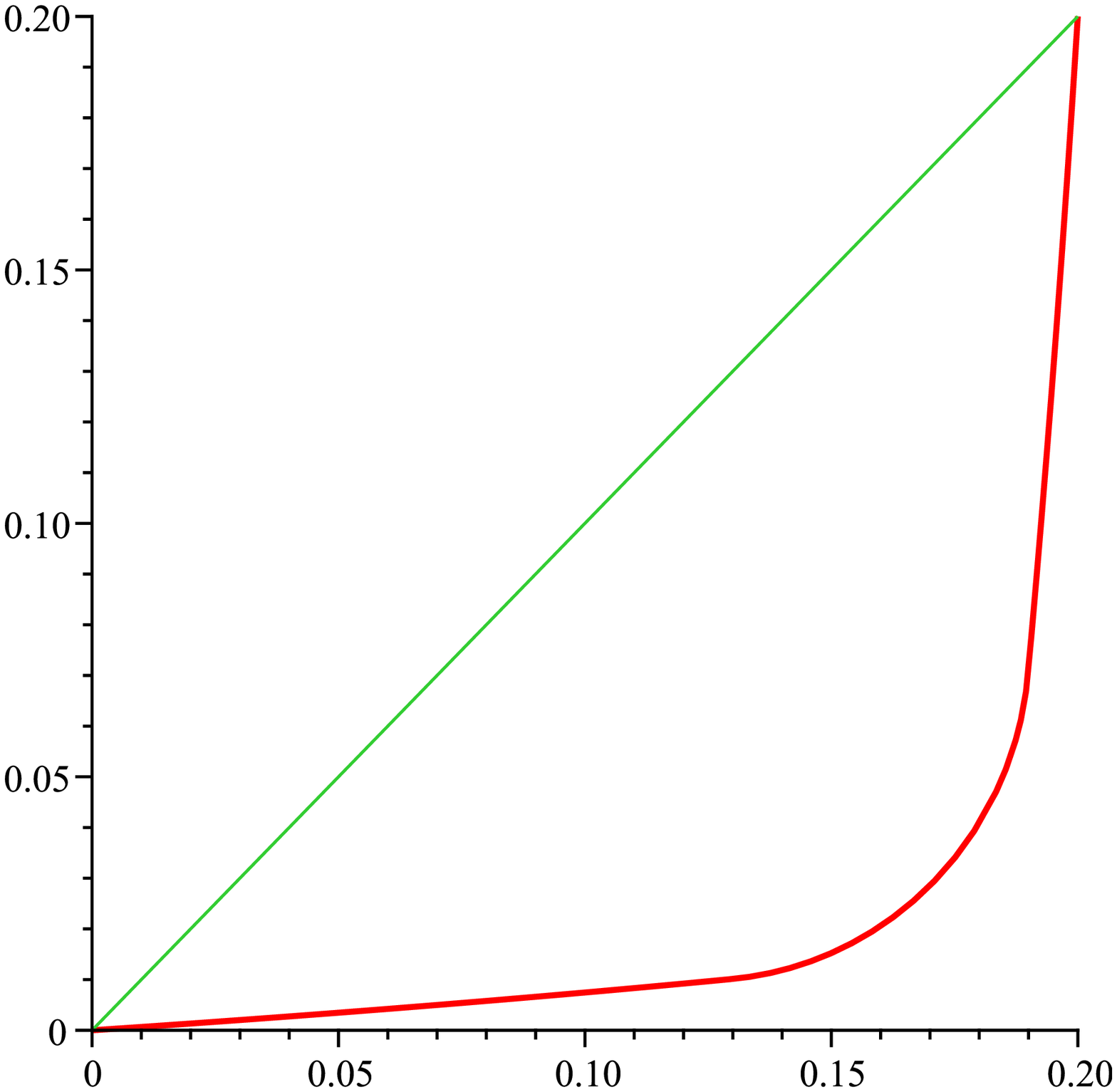}
\caption{Map $\tau_{21}$ for Example \ref{neg_ex}.}
\label{fig:neg_tau21}
\end{minipage}
\end{figure}

\begin{figure}[h]
\begin{minipage}[t]{0.45\linewidth}
\centering
\includegraphics[width=\textwidth]{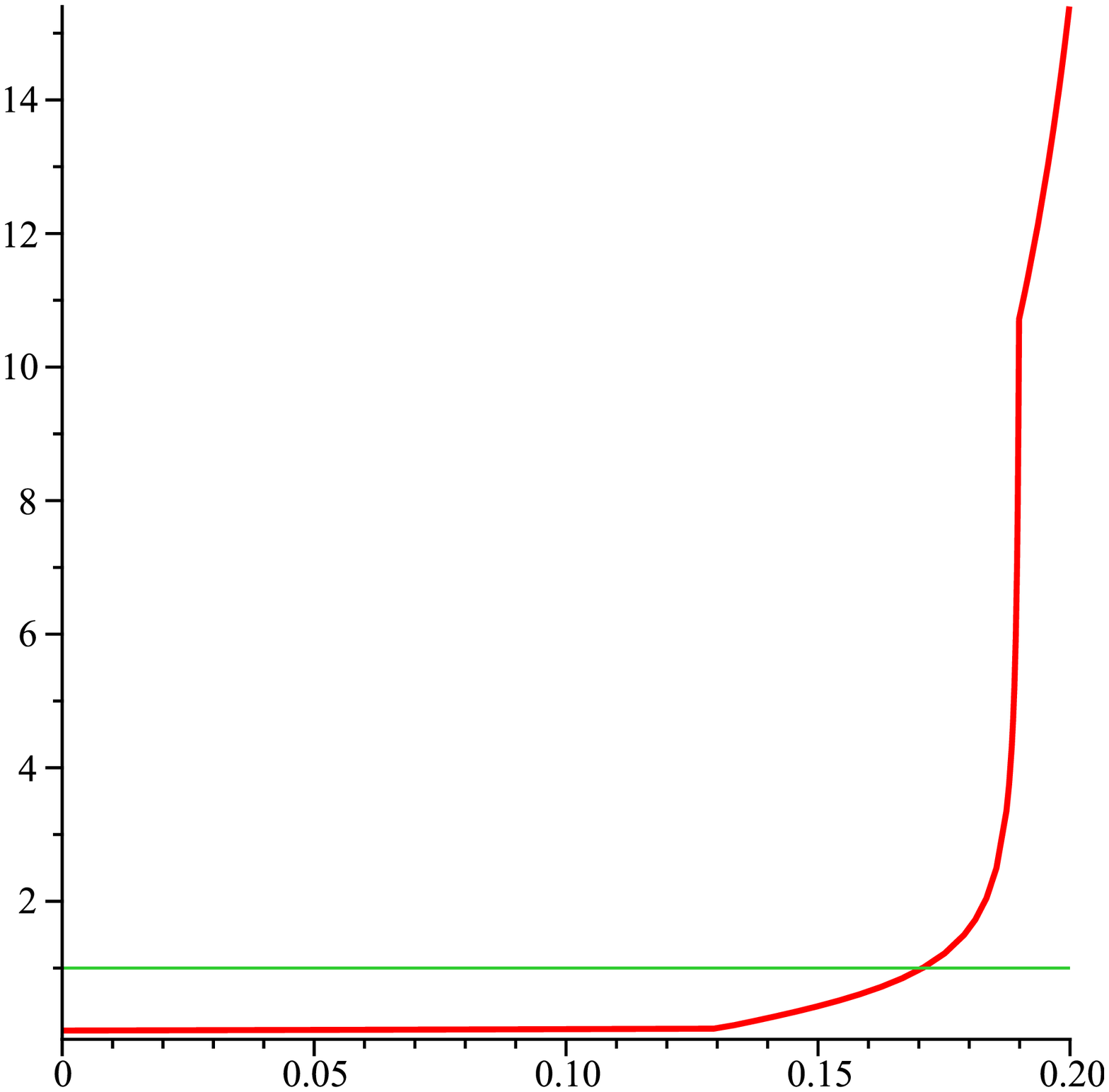}
\caption{Derivative $\tau_{21}'$ for Example \ref{neg_ex}.}
\label{fig:neg_tau21prime}
\end{minipage}
\hspace{0.5cm}
\begin{minipage}[t]{0.45\linewidth}
\centering
\includegraphics[width=\textwidth]{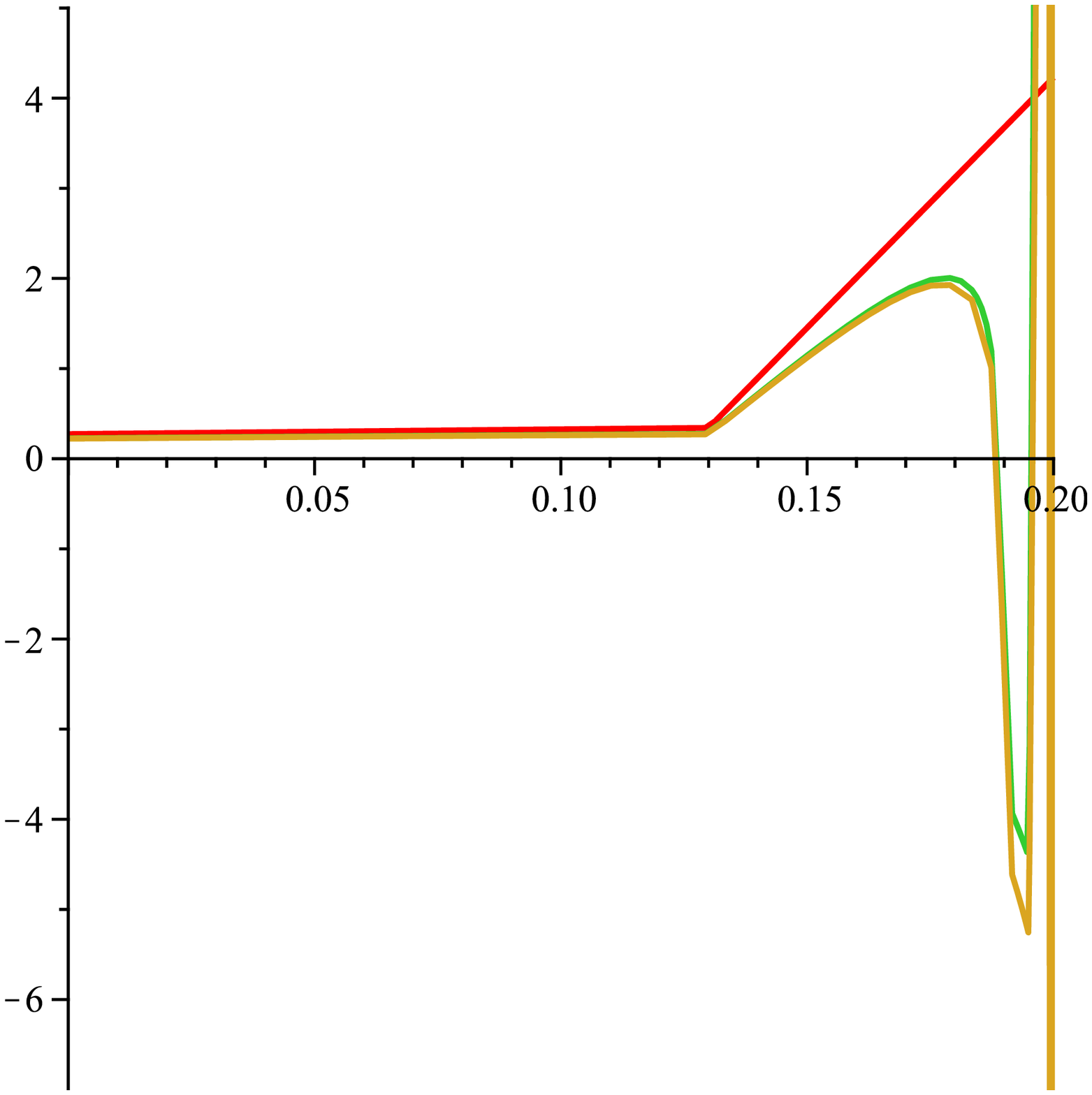}
  \caption{First three iterations of approximating the invariant solution $p$ on $[0,0.2]$, cutting off between -7 and 5 for Example \ref{neg_ex}.}
  \label{fig:neg_approximations20}
\end{minipage}
\end{figure}
\end{example}

\subsection{General density functions}
We now consider the case when the selector $\tau$ has a general pdf.
Let us consider two maps $\tau_1:[0,a]\to[0,1]$ increasing with $\tau_1(0)=0 $, $\tau_1(a)=1 $ satisfying $\tau_1(x)\le x/a$,
and $\tau_{21}:[0,a]\to[0,a]$ increasing with $\tau_{21}(0)=0 $, $\tau_{21}(a)=a $ satisfying $\tau_{21}(x)< x$ on open interval $(0,a)$.
Both $\tau_i$ are extended onto $[0,1]$ by defining them on $(a,1]$ as a monotonic map $\tau:(a,1]\to[0,1]$ which is onto and expanding (i.e. $|\tau'|>1$).

Let $f$ be an invariant density of a selector whose graph is between the graphs of $\tau_1$ and $\tau_2$. We are looking for a probability $p(x)$ such that the random map
$R=\{\tau_1,\tau_2;p, 1-p\}$ preserves the density $f$. The corresponding Frobenius-Perron equation is
\begin{equation}\label{FP_eq11}
 f(x)= \frac{p(\phi_1(x))f(\phi_1(x))}{\tau_1'(\phi_1(x))}+ \frac{(1-p(\psi_1(x)))f(\psi_1(x))}{\tau_2'(\psi_1(x))}+
       \frac{f(\tau^{-1}(x))}{|\tau'(\tau^{-1}(x))|}\ ,
\end{equation}
where $\phi_1=(\tau_1|_{_{[0,a]}})^{-1}$, $\psi_1=(\tau_2|_{_{[0,a]}})^{-1}$, $x\in[0,1]$.
Introducing  $g=p\cdot f$, substituting $x=\tau_1(y)$, $y\in[0,a]$ and using the equality $$\tau_{21}'(y)=\frac {\tau_1'(y)}{\tau_2'(\tau_{21}(y))}\ ,$$
we reduce eq.(\ref{FP_eq11})  to
\begin{equation}\label{main_eq1}
g(y)=g(\tau_{21}(y))\cdot\tau'_{21}(y)+ \left[\left((f(\tau_1(y))-\frac{f(\tau^{-1}(\tau_1(y)))}{|\tau'(\tau^{-1}(\tau_1(y)))|}\right)\cdot\tau_1'(y)-\tau_{21}'(y)\right]\ .
\end{equation}
We want to solve this equation
 hoping that $p$ will satisfy $0<p<1$.

We choose $\sigma>0$ such that $$\tau_{21}'(x)\le \sigma<1\ \ ,\ \ \text{on interval }\ \ [0,\widehat{a}]\ ,\ \widehat{a}<a\ .$$
Let us define the affine operator
$$(\mathcal{P}g)(x)=g(\tau_{21}(y))\cdot\tau'_{21}(y)+ \left[\left((f(\tau_1(y))-\frac{f(\tau^{-1}(\tau_1(y)))}{|\tau'(\tau^{-1}(\tau_1(y)))|}\right)\cdot\tau_1'(y)-\tau_{21}'(y)\right]\ .$$
The most natural space to consider it on is the space $L^\infty[0,a] $ of bounded functions on the interval $[0,a]$.

First, we will consider $\mathcal{P}$ on the smaller space $L^\infty[0,\widehat{a}] $. The proofs of Proposition \ref{prop_gp1} and Corollary \ref{prop_gpc1} are repetition
of those for Propositions \ref{Prop_p1}-\ref{Prop_p4}.



\begin{proposition}\label{prop_gp1}
$\mathcal{P}$ is a contraction on $L^\infty[0,\widehat{a}] $. Its unique fixed point is given by
\begin{equation}\label{sol1}
g(x)=\sum_{n=0}^\infty B(\tau_{21}^n(x))\cdot (\tau_{21}^n)'(x) \ ,
\end{equation}
where $B(y)=\left(f(\tau_1(y))-\frac{f(\tau^{-1}(\tau_1(y)))}{|\tau'(\tau^{-1}(\tau_1(y)))|}\right)\cdot\tau_1'(y)-\tau_{21}'(y) $.
\end{proposition}


\begin{corollary}\label{prop_gpc1}
There exists a unique function $p(x)$ such that eq.(\ref{FP_eq11}) holds. On $[0,\widehat{a})$, $p(x)$ is
given by
\begin{equation}\label{sol2}
p(x)=\frac 1{f(x)}\sum_{n=0}^\infty B(\tau_{21}^n(x))\cdot (\tau_{21}^n)'(x)
\end{equation}
and it can be extended to $[0,a]$ as described in Proposition 4.
\end{corollary}


\section{Singular + Singular =Lebesgue }

In this Section we use the results of Example \ref{quadratic} to construct maps $\tau_1$ and $\tau_2$ that have no acims and
a probability function $p(x)$ such that the random map $R=\left\{\tau_1, \tau_2; p, 1-p\right\}$ preserves the Lebesgue measure.

Let us consider maps $\tau_i:[0,1/2]\to[0,1]$, $i=1,2$, increasing, with $\tau_i(0)=0 $, $\tau_i(1/2)=1 $ satisfying $\tau_1(x)\le 2x$, $\tau_2(x)\ge 2x$
and the map $\tau_{21}=\tau_2^{-1}\circ\tau_1:[0,1/2]\to[0,1/2]$ increasing with $\tau_{21}(0)=0 $, $\tau_{21}(1/2)=1/2 $ satisfying $\tau_{21}(x)< x$ on open interval $(0,1/2)$.

We consider the special case of quadratic $\tau_1$ and $\tau_2$.
Let $$\tau_1(x)=a_1x^2+b_1x\ \ ,\ \ \tau_2(x)=a_2x^2+b_2x\ ,$$
where $a_i=4(1-b_i/2)$, $i=1,2$, $0<b_1<2<b_2<4$. We have
$$\tau_{21}(x)=\frac 1{2(1-b_2/2)}\cdot \left(-b_2/4+\sqrt{b_2^2/16+4(1-b_2/2)[(1-b_1/2)x^2+b_1x/4]}\right)\ .$$

\begin{figure}[htp]
\begin{minipage}[t]{0.45\linewidth}
\centering
\includegraphics[width=\textwidth]{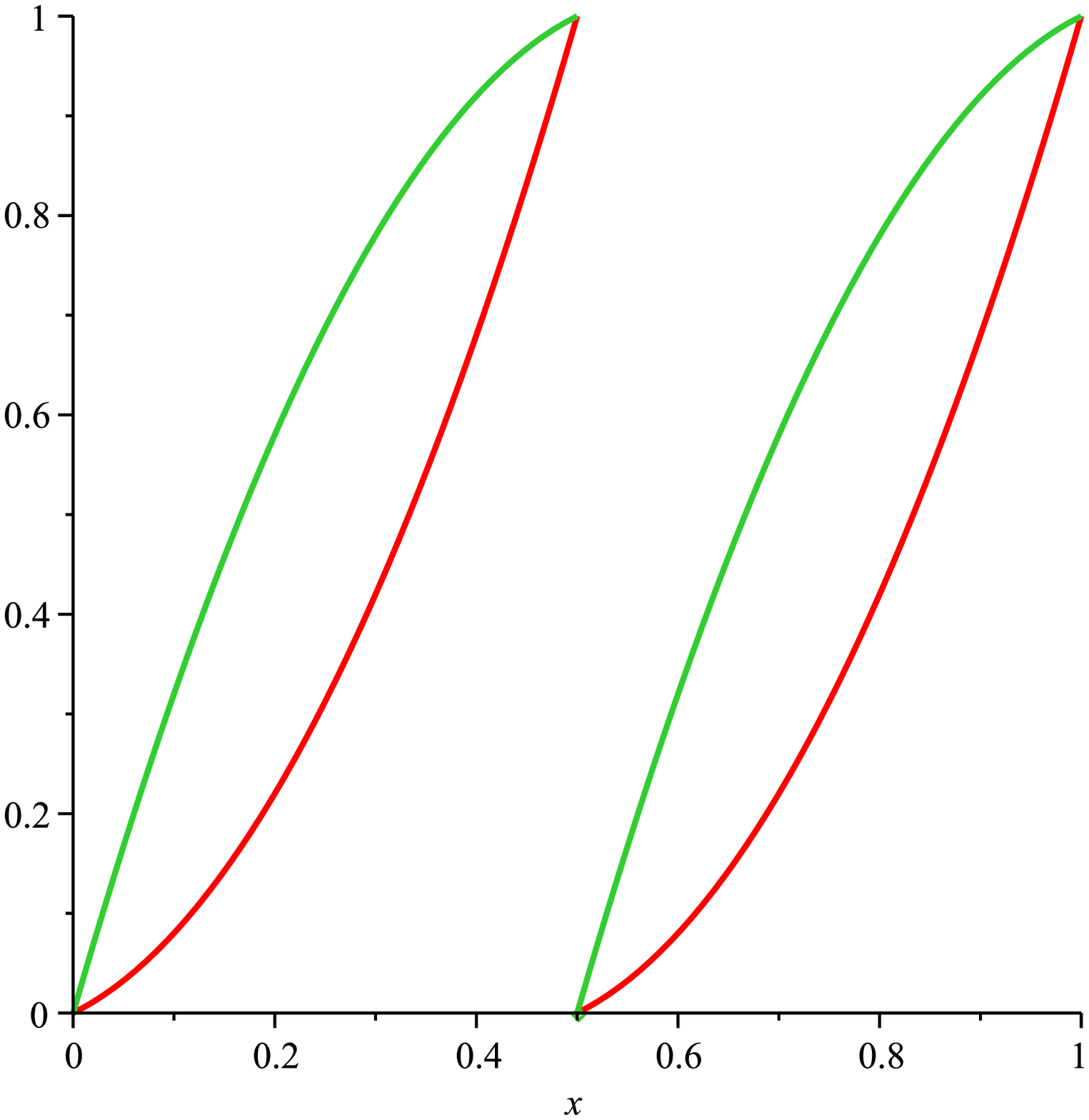}
\caption{Maps $\tau_1$ and $\tau_2$ for $b_1=0.5$, $b_2=3.5$.}
  \label{fig:singular_maps}
\end{minipage}
\hspace{0.5cm}
\begin{minipage}[t]{0.45\linewidth}
\centering
\includegraphics[width=\textwidth]{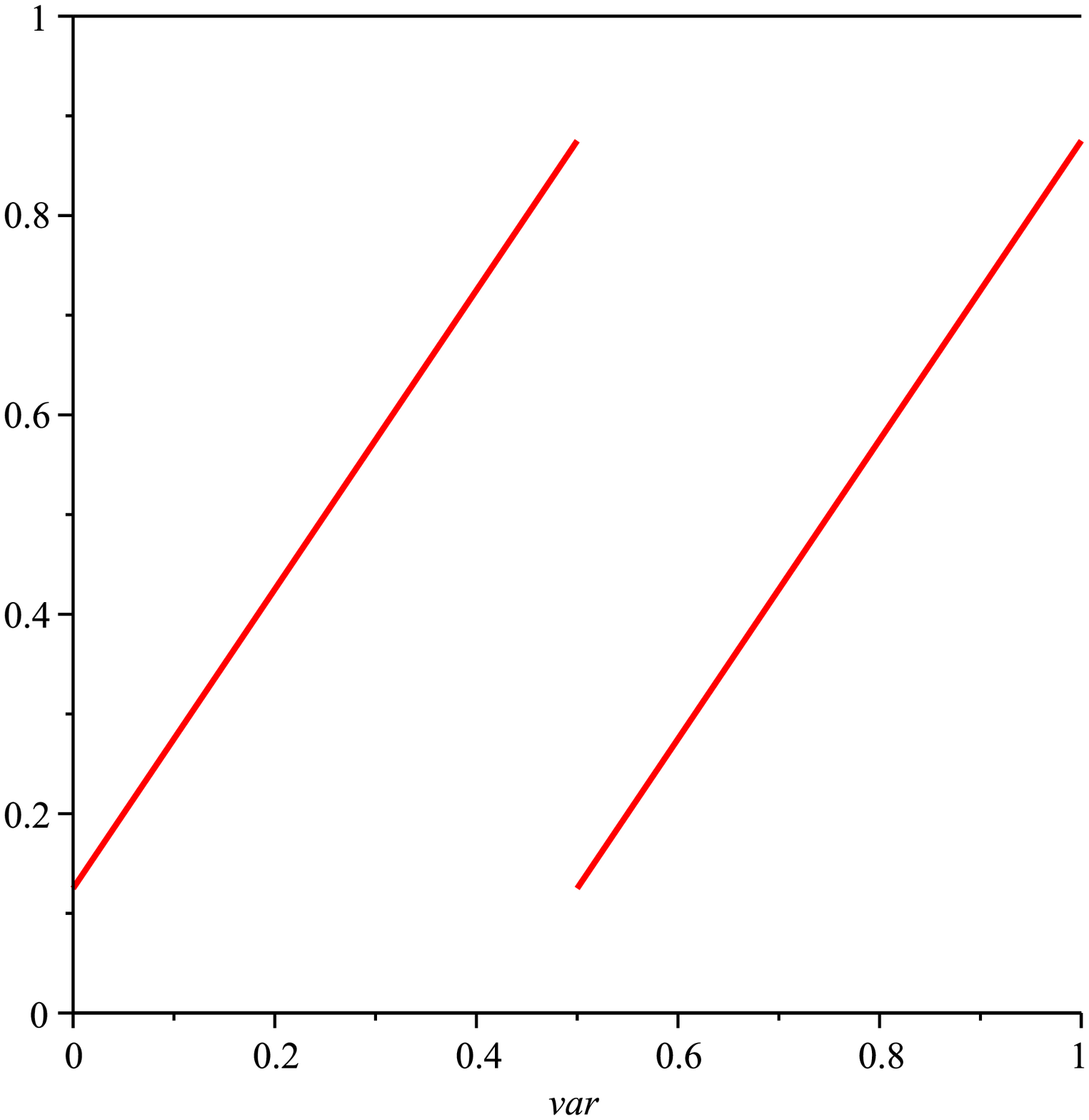}
  \caption{Probability $p(x)$.}
  \label{fig:singular_probability}
\end{minipage}
\end{figure}

We extend the maps $\tau_1$ and $\tau_2$ to the interval $[0,1]$ as follows
$$ \tau_i(x)=\begin{cases} \tau_i(x),& \ \text{for}\ 0\le x\le 1/2\ ;\\
                            \tau_i(x-1/2),& \ \text{for}\ 1/2< x\le 1\ ,
\end{cases}
$$
$i=1,2$.
We want to find a probability $p(x)$ such that the random map $R=\{\tau_1,\tau_2;p, 1-p\}$ preserves Lebesgue measure.
The corresponding Frobenius-Perron equation is
\begin{equation}\label{FP_eq}
 1= \frac{p(\phi_1(x))}{\tau_1'(\phi_1(x))}+ \frac{1-p(\psi_1(x))}{\tau_2'(\psi_1(x))}+
       \frac{p(\phi_1(x)+1/2)}{\tau_1'(\phi_1(x))}+ \frac{1-p(\psi_1(x)+1/2)}{\tau_2'(\psi_1(x))}\ ,
\end{equation}
where $\phi_1=(\tau_1|_{_{[0,1/2]}})^{-1}$, $\psi_1=(\tau_2|_{_{[0,1/2]}})^{-1}$, $x\in[0,1]$.

We will look for a solution $p$ such that $p(x+1/2)=p(x)$ for $x\in[0,1/2]$.
Then the equation (\ref{FP_eq}) reduces to
\begin{equation}\label{FP_eq1}
\frac 12= \frac{p(\phi_1(x))}{\tau_1'(\phi_1(x))}+ \frac{1-p(\psi_1(x))}{\tau_2'(\psi_1(x))}\ ,
\end{equation}
$x\in[0,1]$. Substituting $x=\tau_1(y)$, $y\in[0,1/2]$ and using the equality $$\tau_{21}'(y)=\frac {\tau_1'(y)}{\tau_2'(\tau_{21}(y))}\ ,$$
we reduce it further to the known equation

\begin{equation}\label{main_eq}
p(y)=p(\tau_{21}(y))\cdot\tau'_{21}(y)-\left[\tau_{21}'(y)-(1/2)\cdot\tau_1'(y)\right]\ .
\end{equation}

Using Maple 13 we previously found that the linear function $p(x)=A+Bx$ satisfies (\ref{main_eq}) if we put
$$ A=\frac{b_1(1-b_2/2)}{b_1-b_2}\ \  ,\ \ B=\frac{8(1-b_1/2)(1-b_2/2)}{(b_1-b_2)}\ .$$
Note, that both coefficients are positive as $b_2>b_1$, $b_1/2<1$ and $b_2/2>1$.

For $b_1=0.5$, $b_2=3.5$ neither map $\tau_1$ nor $\tau_2$ has an acim. Almost all trajectories converge to $0$ or $1$ correspondingly. However, the random map $R=\left\{\tau_1, \tau_2; p, 1-p\right\}$, where
$$p(x)=\begin{cases} (3/2)x+1/8,& \ \text{for}\ 0\le x\le 1/2\ ;\\
                            (3/2)(x-1/2)+1/8,& \ \text{for}\ 1/2< x\le 1\ ,
\end{cases}
$$
does preserve Lebesgue measure.

\section{Invariant density of any random map is also an invariant density of a selector}\label{opt_bang}

Let $\tau _{1}$ and $\tau _{2},$ lower and upper maps, respectively, be in $\mathcal{T}$ and possess acims with densities $f_{1}$ and $f_{2}.$  It is shown in \cite{GB03} that $\mathcal{D}$, the set of densities invariant under the
position dependent random maps $R=\left\{\tau _{1},\tau _{2};\ p,1-p\right\},$
is a convex set. Since the underlying space $I=[0,1]$ is compact, the space
of probability measures on $I$ is weakly compact \cite[Theorem 6.4]{P68}. From
Theorem 6 of \cite{GB03} it follows that $\mathcal{D}$ is weakly closed and hence
weakly compact in $L^1$.

Let $BV(I)$ be the space of functions of bounded variation on $I$.
It is proved in \cite{WP05} that if  $\tau _{1}$ and $\tau _{2}$ satisfy the following two conditions
\begin{align*}
\begin{split}
&(A)\, \sum_{k=1}^2 g_k(x)<\alpha<1, \text{ for some } \alpha,\\
&(B)\, g_k\in BV(I), k=1,2,
\end{split}
\end{align*}
where $g_k(x)=\frac{p_k(x)}{|\tau_k'(x)|}$, $k=1,2$,
then the random map $R=\left\{\tau _{1},\tau _{2};\ p,1-p\right\}$ has an invariant density of bounded variation. Since we are considering piecewise expanding maps,
\[\sum_{k=1}^2 g_k(x)<\sum_{k=1}^2 p_k(x)=1.\]
This implies that condition $(A)$ is always satisfied. If we collect all possible probability functions $p$ such that the above condition $(B)$ is satisfied uniformly, then the corresponding set of invariant densities, $\mathcal{D}_{AB}$, is a bounded set in $BV(I)$. Thus, the $L^1$ compactness of $\mathcal{D}_{AB}$ is established.

Consider the convex hull of $f_1$ and $f_2$, denoted by $\mathcal{D}_{Co}=\left\{\alpha f_1+(1-\alpha) f_2|0\leq\alpha\leq 1\right\}$. It follows from Corollary \ref{coro_cov} that we can choose $p$ such that $R=\left\{\tau _{1},\tau _{2};\ p,1-p\right\}$ preserves any $f\in\mathcal{D}_{Co}$. $\mathcal{D}_{Co}$ is closed and compact, has two extreme points,
 $f_1$ and $f_2$, and is a ``line segment" connecting $f_1$ and $f_2$.

We define

$\mathcal{D}_{S}=\left\{f|f \text{ is an invariant pdf of some selector } \boldsymbol\tau\in\mathcal{T}\right\}.$

$\mathcal{D}_{RM}=\left\{f|f \text{ is an invariant pdf of some random map } R=\left\{\tau_{1},\tau_{2};\ p,1-p\right\}\right\}.$

From Examples \ref{bendedfivexmod1} and \ref{neg_ex}, we know
\[\mathcal{D}_{S} \nsubseteq \mathcal{D}_{RM}.\]
By Theorem \ref{Previous_selector_result}, we have
\[\mathcal{D}_{Co}\subset \mathcal{D}_{S}.\]
From Corollary \ref{coro_cov}, it follows that
\[\mathcal{D}_{Co}\subset \mathcal{D}_{RM}.\]

We will now prove that $\mathcal{D}_{RM}\subset\mathcal{D}_{S}$, assuming that we consider only random maps whose pdf's are positive a.e. on $[0,1]$.

\begin{theorem}\label{Th_rand_sel}
Let us consider two maps $\tau_k:[0,1]\to[0,1]$, $k=1,2$, having the same monotonicity on the intervals of monotonicity partition, piecewise expanding (we assume slope $> 2$, or harmonic average of slopes condition \cite{GLBP12}). Then $\mathcal{D}_{RM}\subset\mathcal{D}_{S}$ if we consider only random maps whose pdf's are positive a.e. on $[0,1]$.
\end{theorem}
\begin{proof}
We will use the notations introduced at the beginning of Section \ref{pre_sel}. Recall the definition of extended inverse function in the Section \ref{pre_sel}.  Assuming that $f$ is an invariant density of the random map $R=\left\{\tau_{1},\tau_{2};\ p,1-p\right\}$, the Frobenius-Perron equation for $R$ is
\begin{equation}\label{fp_rand}
f(x)=\sum\limits_{j=1}^m \frac{p(\tau_{1, j}^{-1}(x))f(\tau_{1, j}^{-1}(x))}{\mid\tau'_1(\tau_{1, j}^{-1}(x))\mid}\chi_{_{\tau_{1, j}(I_j)}}(x)+\frac{(1-p(\tau_{2, j}^{-1}(x)))f(\tau_{2, j}^{-1}(x))}{\mid\tau'_2(\tau_{2, j}^{-1}(x))\mid}\chi_{_{\tau_{2, j}(I_j)}}(x).
\end{equation}

Let $\mathbf{F}(x)=\int_0^x f(t)\, dt$, $x\in I$. Now construct a map $\boldsymbol\tau$, piecewise as follows:
\begin{equation}\label{fp_rand1}
\overline{\boldsymbol\tau_j^{-1}}(x)=\mathbf{F}^{-1}\left(\int_0^{\overline{\tau_{1, j}^{-1}}(x)} p(t)f(t)\, dt+\int_0^{\overline{\tau_{2, j}^{-1}}(x)} (1-p(t))f(t)\, dt\right).
\end{equation}
Note that $\boldsymbol\tau$ defined above has the same monotonicity as the boundary maps after the vertical segments are removed. Equation (\ref{fp_rand1}) is equivalent to \[\mathbf{F}\left(\overline{\boldsymbol\tau_j^{-1}}(x)\right)=\int_0^{\overline{\tau_{1, j}^{-1}}(x)} p(t)f(t)\, dt+\int_0^{\overline{\tau_{2, j}^{-1}}(x)} (1-p(t))f(t)\, dt\, .\]
Differentiating both sides of the above equation  with respect to $x$, we have
\begin{equation}
\frac{f(\boldsymbol\tau_{j}^{-1}(x))}{\mid\boldsymbol\tau'(\boldsymbol\tau_{j}^{-1}(x))\mid}\chi_{_{\boldsymbol\tau_{j}(I_j)}}(x)=\frac{p(\tau_{1, j}^{-1}(x))f(\tau_{1, j}^{-1}(x))}{\mid\tau'_1(\tau_{1, j}^{-1}(x))\mid}\chi_{_{\tau_{1, j}(I_j)}}(x)+\frac{(1-p(\tau_{2, j}^{-1}(x)))f(\tau_{2, j}^{-1}(x))}{\mid\tau'_2(\tau_{2, j}^{-1}(x))\mid}\chi_{_{\tau_{2, j}(I_j)}}(x).
\end{equation}
Now, from the definition of the Frobenius-Perron operator of $\boldsymbol\tau$, $P_{\boldsymbol\tau}$, it follows
\begin{align}
\begin{split}
\left(P_{\boldsymbol\tau}f\right)(x)&=\sum\limits_{j=1}^m \frac{f(\boldsymbol\tau_{j}^{-1}(x))}{\mid\boldsymbol\tau'(\boldsymbol\tau_{j}^{-1}(x))\mid}\chi_{_{\boldsymbol\tau_{j}(I_j)}}(x)\\
&=\sum\limits_{j=1}^m\frac{p(\tau_{1, j}^{-1}(x))f(\tau_{1, j}^{-1}(x))}{\mid\tau'_1(\tau_{1, j}^{-1}(x))\mid}\chi_{_{\tau_{1, j}(I_j)}}(x)+\frac{(1-p(\tau_{2, j}^{-1}(x)))f(\tau_{2, j}^{-1}(x))}{\mid\tau'_2(\tau_{2, j}^{-1}(x))\mid}\chi_{_{\tau_{2, j}(I_j)}}(x)\\
&=f(x),
\end{split}
\end{align}
which implies that $\boldsymbol\tau$ preserves $f$.

We have $\overline{\tau_{2, j}^{-1}}(x)\leq \overline{\tau_{1, j}^{-1}}(x)$, so we also have
\begin{align*}
\begin{split}
\mathbf{F}\left(\overline{\boldsymbol\tau_j^{-1}}(x)\right)&\leq \int_0^{\overline{\tau_{1, j}^{-1}}(x)} p(t)f(t)\, dt+\int_0^{\overline{\tau_{1, j}^{-1}}(x)} (1-p(t))f(t)\, dt\\
&=\int_0^{\overline{\tau_{1, j}^{-1}}(x)}f(t)\, dt=\mathbf{F}\left(\overline{\tau_{1, j}^{-1}}(x)\right),
\end{split}
\end{align*}
and
\begin{align*}
\begin{split}
\mathbf{F}\left(\overline{\boldsymbol\tau_j^{-1}}(x)\right)&\geq \int_0^{\overline{\tau_{2, j}^{-1}}(x)} p(t)f(t)\, dt+\int_0^{\overline{\tau_{2, j}^{-1}}(x)} (1-p(t))f(t)\, dt\\
&=\int_0^{\overline{\tau_{2, j}^{-1}}(x)}f(t)\, dt=\mathbf{F}\left(\overline{\tau_{2, j}^{-1}}(x)\right),
\end{split}
\end{align*}
therefore,
\[\mathbf{F}\left(\overline{\tau_{2, j}^{-1}}(x)\right)\leq\mathbf{F}\left(\overline{\boldsymbol\tau_j^{-1}}(x)\right)\leq\mathbf{F}\left(\overline{\tau_{1, j}^{-1}}(x)\right).\]
Then, it follows from the definition of $\mathbf{F}$ that
\[\tau_{1, j}(x)\leq\boldsymbol\tau_{j}(x)\leq\tau_{2, j}(x).\]
This implies that $\boldsymbol\tau$ is a selector. The proof is completed.
\end{proof}

\section{Continuous dependence of invariant densities on position dependent random maps}

The following continuity theorem will be used in the next section.

\begin{theorem}\label{conti_th}
Let us consider two maps $\tau_k:[0,1]\to[0,1]$, $k=1,2$, both piecewise expanding (we assume slope $> 2$, or harmonic average of slopes condition \cite{GLBP12}).
Let us consider a sequence of piecewise semi-Markov approximations $\tau_k^{(n)}$ on uniform partitions $\mathcal P_n$, $\tau_k^{(n)}\to \tau_k$ in $L^1$ as $n\to\infty$.
Let $R^{(n)}=\{\tau_1^{(n)},\tau_2^{(n)}; p_1^{(n)},p_2^{(n)}\}$ be a position dependent random map preserving density $f_{(n)}$. If $p_1^{(n)}\to p_1$ pointwise almost everywhere (actually even weak $L^\infty$ convergence suffices), then  $f_{(n)}\to f$,
where $f$ is invariant for $R=\{\tau_1,\tau_2; p_1,p_2\}$, $p_1+p_2=1$.
\end{theorem}
\begin{proof}
It follows from the slope assumption that the Frobenius-Perron operators $P_{R^{(n)}}$ satisfy the
Lasota-Yorke inequality \cite{BG97,LaY} with uniform constant. And thus, $\left\{f_{(n)}\right\}$ is a precompact set in $L^1$.
We assume that the limit of $\left\{f_{(n)}\right\}$ is $f$.

Let us introduce one more sequence of random maps: $\widehat{R}^{(n)}=\left\{\tau_1,\tau_2; p_1^{(n)},p_2^{(n)}\right\}$,  their invariant densities
are denoted by $\widehat{f}_{(n)}$, $n\in \mathds{N}$.
\begin{align*}
\begin{split}
\parallel P_{R}f-f\parallel_{L^1}\leq&\parallel P_{R}f-P_{\widehat{R}^{(n)}}f\parallel_{L^1}
+\parallel P_{\widehat{R}^{(n)}}f-P_{R^{(\, ,n)}}f_{(n)}\parallel_{L^1}\\
&+
\parallel P_{\widehat{R}^{(n)}}f_{(n)}-P_{R^{(n)}}f_{(n)}\parallel_{L^1}+
\parallel P_{R^{(n)}}f_{(n)}-f\parallel_{L^1}\\
=&\parallel \sum_{k=1}^2P_{\tau_k}p_kf-\sum_{k=1}^2P_{\tau_k}p^{(n)}_kf\parallel_{L^1}+
\parallel P_{\widehat{R}^{(n)}}\left(f-f_{(n)}\right)\parallel_{L^1}\\
&+
\parallel P_{\widehat{R}^{(n)}}f_{(n)}-P_{R^{(n)}}f_{(n)}\parallel_{L^1}+
\parallel f_{(n)}-f\parallel_{L^1}\\
\leq&\sum_{k=1}^2\parallel (p_k-p^{(n)}_k)f\parallel_{L^1}+
\parallel f-f_{(n)}\parallel_{L^1}\\
&+
\parallel P_{\widehat{R}^{(n)}}f_{(n)}-P_{R^{(n)}}f_{(n)}\parallel_{L^1}+
\parallel f_{(n)}-f\parallel_{L^1}.
\end{split}
\end{align*}
The first summand is sufficiently small when $n$ large enough since $p^{(n)}_k\rightarrow p_k$.
The second and fourth terms approach 0 by the definition.
For the third term, note that $R^{(n)}$ and $\widehat{R}^{(n)}$ have the same probability
functions, we also have $\tau_k^{(n)}\to \tau_k$ as $n\to\infty$, it follows from the stability of $\widehat{R}^{(n)}$
that this term also approaches 0 as $n\to\infty$.
This completes the proof.
\end{proof}
\begin{remark}
We need the assumption that slope $> 2$ or harmonic average of slopes condition for $\tau_k$. For example,
$\tau_1^{(n)}=W_0$, $\tau_2^{(n)}=W_{a_n}$, where $W_0$ is the original $W-$shaped map and $W_{a_n}$ is a sequence of $W-$shaped maps of $W_0$ being perturbed in a neighborhood of $1/2$, call it $U_{(1/2)}$ (The detailed construction of $W_{a_n}$ can be found in \cite{LGBPE13}). We put $p_2^{(n)}=1$ in the set $U_{(1/2)}$ . Then as $R^{(n)}=\{W_0,W_{a_n}; p_1^{(n)},p_2^{(n)}\}\rightarrow R=\{W_0,W_0; p_1,p_2\}$, $f_{(n)}$ shall converge to a measure with singular component. However, $R$ preserves some acim.
\end{remark}

\section{Optimization and extreme points}

A well known formulation of the Krein-Milman theorem states that every linear
continuous functional on a locally convex Hausdorff linear space $E$ attains
its minimum on a compact subset $X$ of $E$ at some extreme point of $X$ \cite{B53}. This
result does not require compactness of $X$ and has been extended to lower
semi-continuous concave functionals in \cite{B58}.

In general we do not know the extreme points of $\mathcal{D}_{RM}$. However, in the special case
where $\tau _{1}$ and $\tau _{2}$ are semi-Markov piecewise linear, it has
been shown in \cite{GB06} that the extreme points of $\mathcal{D}_{RM}$ come from the
random maps where, for each $x\in I$, $p(x)\in \{0,1\},$ that is,
deterministic maps, taking their values on either the lower or upper
boundaries. We call such a random map a bang-bang map. In this setting, we
consider a continuous linear functional on $\mathcal{D}_{RM}$ which is to be
optimized. By the Krein-Milman theorem the random map that optimizes the
functional over the admissible probability density functions is a bang-bang
map.

A typical continuous linear functional that appears in optimization problems
is of the form:
\[
F(f)=\int_{0}^{1}g(x)f(x)dx,
\]
where $g$ is a fixed bounded function on $I$ and $f$ is a density function
in $\mathcal{D}_{RM}$.

Question: Can we generalize the above bang-bang result to more general maps?

From now on, let us consider the real space $L^1_I$ of Lebesgue integrable function on $I$. It is a locally convex, linear Hausdorff space. For a sequence of closed sets of
$L^1_I$, $\left\{\mathcal{S}_n\right\}$, we write $\limsup\mathcal{S}_n$ to refer to the limit superior in the sense of Kuratowski convergence \cite{K66}.
\begin{definition}
The point $\mathbf{s}$ belongs to $\limsup\mathcal{S}_n$, if every neighbourhood of $\mathbf{s}$ intersects an infinite number of the $\mathcal{S}_n$.
\end{definition}

We will also need the following theorem from \cite[Theorem 2]{J54}, which we have modified for our use.
\begin{theorem}\label{J_Th2}
Let $\left\{\mathcal{S}_n\right\}$ be a sequence of compact convex sets whose
union is contained in a compact convex subset of $L^1_I$. Let $\mathcal{X}_n$ be the set of
extreme points of $\mathcal{S}_n$ for each $n$ and let $\mathcal{S}=\limsup\mathcal{S}_n$ and $\mathcal{X}
=\limsup\mathcal{X}_n$. Then $\mathcal{S}\subset Co(\mathcal{X})$, where $Co(\mathcal{X})$ is the convex hull of
$\mathcal{X}$. If, in addition, $\mathcal{S}$ is convex, then $\mathcal{S}=Co(\mathcal{X})$,
that is, $\mathcal{X}$ contains all the extreme points of $\mathcal{S}$.
\end{theorem}

Our boundary maps are assumed to be piecewise expanding, it follows from Pelikan's condition \cite{P84} that the
  random map $R=\left\{\tau_{1},\tau_{2};\ p,1-p\right\}$ always has an invariant density function in $L^1_I$ for any measurable
  probability function $p$. We introduce two sequences of semi-Markov maps, $\left\{\tau^N_{1}\right\}$ and $\left\{\tau^N_{2}\right\}$, which approximate $\tau_{1}$ and $\tau_{2}$ respectively on uniform
partitions of $I$ into $N$ equal subintervals $\left\{I^u_{i}\right\}^N_{i=1}$. Moreover, we approximate the probability function $p$ by a sequence of step functions on these subintervals, denoted by $\left\{p^N\right\}$, i.e. $p^N=\sum^N_{i=1}p^N_i\chi_{_{I^u_i}}$. We can do this in such way that $p^N\rightarrow p$ pointwise almost everywhere.

We define sets of attainable densities and consider them as subsets of $L^1$:
\begin{align*}
\begin{split}
\mathcal{A}&=\left\{f:\text{ there exits probability }p \text{ such that }f\text{ is an invariant density of }R=\left\{\tau_{1},\tau_{2};\ p,1-p\right\}\right\},\\
\mathcal{A}_N&=\Big\{f_N:\text{ there exits probability }p^N \text{ which is piecewise constant on the uniform partition }\\
&\quad\quad\left\{I^u_{i}\right\}^N_{i=1} \text{  such that }f_N\text{ is an invariant density of }R^N=\left\{\tau^N_{1},\tau^N_{2};\ p^N,1-p^N\right\}\Big\},\\
\mathcal{E}_N&=\left\{f:f \text{ is an extreme point in }\mathcal{A}_N\right\},\\
\mathcal{E}&=\limsup\mathcal{E}_N, 
\end{split}
\end{align*}
where $N\in\mathds{N}$. Note that each $\mathcal{A}_N$ is defined for fixed $\tau^N_{1}$ and $\tau^N_{2}$, and $p^N$ can be viewed as a vector which varies.

\begin{remark}\label{rem_convex}
(I) $\mathcal{A}_N$ is convex \cite{GB06}.  Since $\tau^N_{1}$ and $\tau^N_{2}$ have large
enough slope, $\mathcal{A}_N$ is set of functions of bounded variations, and therefore is precompact. Furthermore, it follows from Lemma 1 and Theorem 6 in \cite{GB03} that, $\mathcal{A}_N$ is closed. Therefore $\mathcal{A}_N$ is compact.

(II) Each $\mathcal{A}_N$ is contained in a set of functions of bounded variations and it contains all density functions which are  the approximations of the invariant density functions of the random map $R=\left\{\tau_{1},\tau_{2};\ p,1-p\right\}$ for some $p$. Thus $\left\{\mathcal{A}_N\right\}$ is contained in some closed set of functions of bounded variations, which is compact.

(III) Now, consider a density function $f$ which is invariant for the random map $R=\left\{\tau_{1},\tau_{2};\ p,1-p\right\}$ for some $p$. It follows from Theorem \ref{conti_th} that we can choose a sequence of maps $R^N=\left\{\tau^N_{1},\tau^N_{2};\ p^N,1-p^N\right\}$, such that
\[f^N\in\mathcal{A}_N,\ \tau^N_{1}\rightarrow \tau_1,\  \tau^N_{2}\rightarrow \tau_2,\ p^N\rightarrow p,\]
where $f^N$ is the invariant density of $R^N$, and $f^N$ converges to the invariant density, $f$, of $R$. The convergence implies every neighbourhood of $f$ meets the sets of the sequence $\left\{\mathcal{A}_N\right\}$ with arbitrarily large index $N$. Thus, \[\limsup\mathcal{A}_N=\mathcal{A}.\]

(IV) $\mathcal{A}$ is convex \cite{GB03,GB06}.
\end{remark}

It follows from Theorem \ref{J_Th2} and (I-IV) in the Remark \ref{rem_convex} that:
\begin{theorem}
$Co(\mathcal{E})=\mathcal{A}$, where $Co(\mathcal{E})$ is the convex hull of $\mathcal{E}$.
\end{theorem}
This means $\mathcal{E}$ contains all the extreme points of $\mathcal{A}$.
\begin{corollary}\label{coro_extre}
Every invariant density which is an extreme point in $\mathcal{A}$ is a limit of densities $f^N$ which are invariant
for some random map $R^N$ with probability $p^N(x)\in\{0,1\}$.
\end{corollary}

Corollary \ref{coro_extre} suggests the following algorithm to optimize a continuous functional $F$ on the
set of densities $\mathcal{D}_{RM}$. For large enough $N$, construct invariant densities $f^N_k$ for all
``bang-bang" probabilities $p^N_k$. Find $f^N_k$ which optimizes $F$ on  $\mathcal{D}^N_{RM}$. Then, $\left\{F(f^N_k)\right\}$ converges to the optimal value of $F$ on $\mathcal{D}_{RM}$.

\section{Summary of examples}
(i) Example 1 presents a positive general solution for the case when the selector is the triangle map and the boundary maps have two branches, the first quadratic, the second linear.

(ii) Example 2 presents a positive solution for the specific case of a selector and  semi-Markov  boundary maps
with five branches.

(iii) Example 3 shows that with the same boundary maps as in Example 2 and another selector (which has an invariant pdf), it is impossible.

(iv) Example 4  demonstrates a  successful application of Propositions 1-4 to find the probability function defining
a random map with the required pdf.

(v) Example 5 shows that the method of Propositions 1-4 sometimes fails. The function produced is not a probability function.

(vi) In Section 5 we produce boundary maps that have no acim, but a random map based on them preserves Lebesgue measure invariant for a selector.

\section{Concluding remarks}
The main objective of this paper is to study conditions for a selector of a multivalued function with graph G to have statistical dynamics that can be represented by the dynamics of a position dependent random map based solely on the boundaries of G. Theorem 1 and Corollary 1 state that any convex combination of pdf's of boundary maps can be realized both as the pdf of a selector and as the pdf of a random map. We develop results that work in general but demonstrate with examples that some of our methods depend very sensitively on the selector and on the boundary maps.

We also study the converse problem: when can a pdf of a random map based on the boundary maps be realized as a pdf of some selector? Theorem \ref{Th_rand_sel} proves that this always holds in the case of piecewise expanding boundary maps which are piecewise monotonic having the same monotonicity on the partition intervals. Finally we study the extreme points of the set of pdf's of all random maps based on the boundary maps and attempt to characterize them.

In the future we plan to generalize Theorem \ref{Th_rand_sel} omitting, if possible, the assumption
of common monotonicity on partition intervals. A major project will be the generalization of our results to higher dimensions.

\begin{acknowledgement}
We are extremely grateful for the very helpful comments of the anonymous reviewers. Their suggestions and critiques have greatly improved the paper.
\end{acknowledgement}

\end{document}